%% file: AssociatedOPs.tex
\documentclass{article}

\input{header}

\begin{document}

\title{Fast associated classical orthogonal polynomial transforms}
% Short title for running heads:

\author{%
{\sc Brock Klippenstein and Richard Mika\"el Slevinsky\thanks{Corresponding author. Email: Richard.Slevinsky@umanitoba.ca}}\\[2pt]
Department of Mathematics, University of Manitoba, Winnipeg, Canada}
% Short list of authors for running heads:

\maketitle

\begin{abstract}
We discuss a fast approximate solution to the associated classical -- classical orthogonal polynomial connection problem. We first show that associated classical orthogonal polynomials are solutions to a fourth-order quadratic eigenvalue problem with polynomial coefficients such that the differential operator is degree-preserving. Upon linearization, the discretization of this quadratic eigenvalue problem is block upper-triangular and banded. After a perfect shuffle, we extend a divide-and-conquer approach to the upper-triangular and banded generalized eigenvalue problem to the blocked case, which may be accelerated by one of a few different algorithms. Associated orthogonal polynomials arise from iterated Stieltjes transforms of orthogonal polynomials; hence, fast approximate conversion to classical cases combined with fast discrete sine and cosine transforms provides a modular mechanism for synthesis of singular integral transforms of classical orthogonal polynomial expansions.
\end{abstract}

\section{Introduction}

Let $L^2(D,\ud\mu)$ denote the Hilbert space of square integrable functions on $D=\supp(\mu)\subset\R$ with positive Borel measure $\ud\mu$ with inner-product $\displaystyle\langle f,g\rangle = \int_Dfg\ud\mu$. We shall denote by $\{p_n(x)\}_{n=0}^\infty$ an orthogonal polynomial sequence. Orthogonal polynomials satisfy a three-term recurrence relation which may be cast in the following form:
\begin{equation}\label{eq:COPREC}
p_{n+1}(x) = (A_nx+B_n)p_n(x) - C_np_{n-1}(x),\qquad p_0(x) = 1, \qquad p_{-1}(x) = 0,
\end{equation}
where $A_{n-1}A_nC_n>0$ for $n\ge1$.

Given a family of orthogonal polynomials, the {\em associated} orthogonal polynomials are those polynomials with the same initial conditions that use the recurrence coefficients with indices offset by $c\in\mathbb{N}_0$ units:
\begin{equation}
p_{n+1}(x;c) = (A_{n+c}x+B_{n+c})p_n(x;c) - C_{n+c}p_{n-1}(x;c).
\end{equation}
Thanks to Favard's theorem~\cite{Favard-200-2052-35}, the associated polynomials are indeed orthogonal with respect to non-negative measures $\mu(x;c)$ with $D=\supp(\mu(\cdot,c))\subset\R$.

Associated orthogonal polynomials diagonalize the following integral transforms~\cite[vol. 2, p. 162 (6)]{Erdelyi-et-al-2-53} with a removable singularity:
\begin{equation}\label{eq:STSR}
p_{n-1}(x;c+1) = \frac{1}{A_c\int_D \mathrm{d}\mu(t;c)}\int_D\frac{p_n(x;c)-p_n(t;c)}{x-t}{\rm\,d}\mu(t;c).
\end{equation}

Classical orthogonal polynomials are characterized by Bochner~\cite{Bochner-29-730-29} and Krall~\cite{Krall-4-705-38} as the polynomial solutions of the degree-preserving second-order linear differential equation:
\begin{equation}\label{eq:COPEIG}
\left(\sigma \DD^2 + \tau \DD\right) p_n = \lambda_n p_n,
\end{equation}
where $\sigma$ and $\tau$ are polynomials in $x$ independent of $n$ that satisfy $\deg(\sigma) \le 2$ and $\deg(\tau) \le 1$, and the eigenvalues $\lambda_n = \frac{n}{2}\left[(n-1)\sigma''+2\tau'\right]$. We call this factorization degree-preserving because the degrees of the polynomial variable coefficients do not exceed the orders of the respective differential operators. Table~\ref{table:COP} summarizes this classical characterization data.

\begin{table}
\begin{center}
\caption{Classical Orthogonal Polynomial data for Eqs.~\eqref{eq:COPREC} and~\eqref{eq:COPEIG}.}
\label{table:COP}
\begin{tabular}{c|ccc}
\sphline
Name & Jacobi & Laguerre & Hermite\\
\sphline
Hilbert space & $L^2([-1,1], (1-x)^\alpha(1+x)^\beta\ud x)$ & $L^2(\R^+, x^\alpha e^{-x}\ud x)$ & $L^2(\R, e^{-x^2}\ud x)$\\
Special notation & $P_n^{(\alpha, \beta)}(x)$ & $L_n^{(\alpha)}(x)$ & $H_n(x)$\\
$A_n$ & $\dfrac{(2n+\alpha+\beta+1)(2n+\alpha+\beta+2)}{2(n+1)(n+\alpha+\beta+1)}$ & $-\dfrac{1}{n+1}$ & $2$\\
$B_n$ & $\dfrac{(\alpha^2-\beta^2)(2n+\alpha+\beta+1)}{2(n+1)(n+\alpha+\beta+1)(2n+\alpha+\beta)}$ & $\dfrac{2n+\alpha+1}{n+1}$ & $0$\\
$C_n$ & $\dfrac{(n+\alpha)(n+\beta)(2n+\alpha+\beta+2)}{(n+1)(n+\alpha+\beta+1)(2n+\alpha+\beta)}$ & $\dfrac{n+\alpha}{n+1}$ & $2n$\\
$h_n = \langle p_n, p_n\rangle$ & $\dfrac{2^{\alpha+\beta+1}\Gamma(n+\alpha+1)\Gamma(n+\beta+1)}{(2n+\alpha+\beta+1)\Gamma(n+\alpha+\beta+1)n!}$ & $\dfrac{\Gamma(n+\alpha+1)}{n!}$ & $\sqrt{\pi}2^nn!$\\
$\sigma$ & $x^2-1$ & $-x$ & $-1$\\
$\tau$ & $\alpha-\beta + (\alpha + \beta + 2)x$ & $x-\alpha-1$ & $2x$\\
$\lambda_n$ & $n(n+\alpha + \beta +1)$ & $n$ & $2n$\\
\sphline
\end{tabular} 
\end{center}
\end{table}

There is an appreciable literature on the associated classical orthogonal polynomial differential equation. First, it is built from special cases for Hermite and Laguerre polynomials~\cite{Askey-Wimp-96-15-84}. This is followed upon by the Jacobi polynomial case~\cite{Wimp-39-983-87} for a complete classical picture. The differential equation is linear and fourth-order with polynomial variable coefficients that depend on both $x$ and $n$; hence the associated classical polynomials are members of the Laguerre--Hahn family~\cite{Laguerre-1-135-85,Hahn-21-1-78}.

Zarzo, Ronveaux, and Godoy~\cite{Zarzo-Ronveaux-Godoy-49-349-93} present the degree-preserving differential equation in terms of some of the classical orthogonal polynomial data, notably $\sigma$ and $\tau$. It reads:
\begin{align}
& \left\{-\sigma^2 \DD^4 - 5\sigma \sigma' \DD^3\right.\nonumber\\
& + \left[\tau^2-2\tau\sigma'-3\sigma'^2+(2n+4c)\sigma\tau'-(4+n-n^2+4c-2nc-2c^2)\sigma\sigma''\right] \DD^2\nonumber\\
& \left. + \tfrac{3}{2}\left[2\tau\tau' + (2n-2+4c)\sigma'\tau'-2\tau\sigma'' + (n^2-n-4c+2nc+2c^2)\sigma'\sigma''\right]\DD\right\}p_n(x;c)\label{eq:ACOPEIG1}\\
& = \tfrac{1}{4} n(n+2)\left[(n+2c-3)\sigma''+2\tau'\right]\left[(n+2c-1)\sigma''+2\tau'\right]p_n(x;c)\nonumber.
\end{align}

Foupouagnigni, Koepf, and Ronveaux~\cite{Foupouagnigni-Koepf-Ronveaux-162-299-04} refine this presentation by including the classical eigenvalues in the description:
\begin{align}
& \left\{-\sigma^2 \DD^4 - 5\sigma \sigma' \DD^3\right.\nonumber\\
& + \left[\tau^2+2\tau'\sigma-2\tau\sigma'-6\sigma\sigma''+2(\lambda_{n+c}+\lambda_{c-1})\sigma - 3\sigma'^2\right] \DD^2\nonumber\\
& \left. + 3\left[\tau\tau' + (\lambda_{n+c}+\lambda_{c-1})\sigma' - (\tau+\sigma')\sigma''\right] \DD\right\}p_n(x;c)\label{eq:ACOPEIG2}\\
& = \left[(\lambda_{n+c}-\lambda_{c-1})^2-(\lambda_{n+c}+\lambda_{c-1})\sigma'' + \tau'(\sigma''-\tau')\right] p_n(x;c).\nonumber
\end{align}
In this way, Eq.~\eqref{eq:ACOPEIG2} is the associated analogue of Eq.~\eqref{eq:COPEIG}.

There has also been significant interest in factorizing the fourth-order differential equation. For the first order of association, $c=1$, Ronveaux~\cite{Ronveaux-21-L749-88} presents a homogeneous degree-preserving factorization of Eq.~\eqref{eq:ACOPEIG2}:
\begin{equation}
\left[\sigma\DD^2 + (\tau+\sigma')\DD + \tau' - \lambda_{n+1}\right]\left[\sigma\DD^2+(2\sigma'-\tau)\DD + \sigma''-\tau'-\lambda_{n+1}\right] p_n(x;1) = 0.\label{eq:ACOPEIG3a}
\end{equation}
More generally, Lewanowicz~\cite{Lewanowicz-65-215-95} shows that the same differential operators in Eq.~\eqref{eq:ACOPEIG3a} factorize Eq.~\eqref{eq:ACOPEIG2} incompletely:
\begin{align}
& \left[\sigma\DD^2 + (\tau+\sigma')\DD + \tau' - \lambda_{n+1}\right]\left[\sigma\DD^2+(2\sigma'-\tau)\DD + \sigma''-\tau'-\lambda_{n+1}\right] p_n(x;c)\nonumber\\
& = (c-1)\left[(n+c-1)\sigma''+2\tau'\right]\left[2\sigma\DD^2+3\sigma'\DD-n(n+2)\sigma''\right] p_n(x;c).\label{eq:ACOPEIG3}
\end{align}
Foupouagnigni, Koepf, and Ronveaux~\cite{Foupouagnigni-Koepf-Ronveaux-162-299-04} also present a factorization of the fourth-order differential equation that is not degree-preserving, as the variable coefficients are also described by the classical orthogonal polynomials:
\begin{align}
& \left\{\sigma p_{c-1}\DD^2 + \left[(\tau+\sigma')p_{c-1}-2\sigma p_{c-1}'\right]\DD - \left[(\lambda_{n+c}-\lambda_{c-1}-\tau')p_{c-1}-2\tau p_{c-1}'\right]\right\}\nonumber\\
& \times \left\{\sigma p_{c-1}^2\DD^2 - p_{c-1}\left[(\tau-2\sigma')p_{c-1}+2\sigma p_{c-1}'\right]\DD \right.\label{eq:ACOPEIG4}\\
& \quad \left. + \left[2(\tau-\sigma')p_{c-1}p_{c-1}' + 2\sigma p_{c-1}'^2 + (\sigma''-\tau'-\lambda_{n+c}-\lambda_{c-1})p_{c-1}^2\right] \right\} p_n(x;c) = 0;\nonumber
\end{align}

We shall require a refinement of Eq.~\eqref{eq:ACOPEIG2}.
\begin{theorem}\label{theorem:ACOPQEP}
Let:
\begin{align}
\AA & = -\sigma^2 \DD^4 - 5\sigma \sigma' \DD^3 + \left[\tau^2+2\tau'\sigma-2\tau\sigma'-6\sigma\sigma''+4\lambda_{c-1}\sigma - 3\sigma'^2\right] \DD^2\nonumber\\
& \quad + 3\left[\tau\tau' + 2\lambda_{c-1}\sigma' - (\tau+\sigma')\sigma''\right] \DD + \left[2\lambda_{c-1}\sigma'' - \tau'(\sigma''-\tau')\right]\II,\quad and\\
\BB & = 2\sigma\DD^2 + 3\sigma'\DD + \sigma''\II.
\end{align}
The quadratic eigenvalues $\{\lambda_n^\pm\}_{n=0}^\infty$ of the problem defined by:
\begin{equation}\label{eq:ACOPEIG2a}
\{\AA + \lambda_n^\pm\BB\}p_n^\pm(x;c) = (\lambda_n^\pm)^2p_n^\pm(x;c).
\end{equation}
satisfy:
\begin{align}
2\lambda_n^\pm & = \sigma''(n+1)^2 \pm (n+1)\sqrt{(\sigma''-2\tau')^2+8\lambda_{c-1}\sigma''},\label{eq:QuadraticSpectrum1}\\
& = \pm(n+1)\left\{\left[2c-3\pm(n+1)\right]\sigma''+2\tau'\right\}.\label{eq:QuadraticSpectrum2}
\end{align}
In particular, it follows from $\lambda_n^+ = \lambda_{n+c}-\lambda_{c-1} > 0$ that $p_n^+(x;c) = p_n(x;c)$, the associated classical orthogonal polynomials, and Eq.~\eqref{eq:ACOPEIG2} is alternatively formulated as a quadratic eigenvalue problem with known eigenvalues.
\end{theorem}
\begin{proof}
As in the classical case of Eq.~\eqref{eq:COPEIG}, the eigenvalues of Eq.~\eqref{eq:ACOPEIG2a} are found by applying $\AA$ and $\BB$ to $x^n$ and extracting the leading coefficient. We find:
\[
-\frac{\sigma''^2n(n+1)^2(n+2)}{4} + \left(\tau'^2-\tau'\sigma''+2\lambda_{c-1}\sigma''\right)(n+1)^2 + \sigma''(n+1)^2 \lambda_n^\pm = (\lambda_n^\pm)^2.
\]
Eq.~\eqref{eq:QuadraticSpectrum1} solves this quadratic equation, and when we write $\lambda_{c-1}$ in terms of $c$, $\sigma''$ and $\tau'$, we find that the discriminant is a perfect square, concluding with the alternative form in Eq.~\eqref{eq:QuadraticSpectrum2}. Finally, it is readily confirmed that:
\begin{align*}
\lambda_n^+ & = \frac{(n+1)}{2}\left[(n+2c-2)\sigma''+2\tau'\right],\\
& = \frac{(n+c-c+1)}{2}\left[(n+2c-2)\sigma''+2\tau'\right],\\
& = \frac{(n+c)}{2}\left[(n+c-1+c-1)\sigma''+2\tau'\right] - \frac{(c-1)}{2}\left[(n+c+c-2)\sigma''+2\tau'\right],\\
& = \frac{(n+c)}{2}\left[(n+c-1)\sigma''+2\tau'\right] - \frac{(c-1)}{2}\left[(c-2)\sigma''+2\tau'\right] = \lambda_{n+c}-\lambda_{c-1}.
\end{align*}
\end{proof}
The full spectrum of Eq.~\eqref{eq:ACOPEIG2a} is helpful in determining if the positive family of polynomial eigenfunctions is linearly independent of the negative family. As degree-graded polynomials form a flag, it is also the case for each set of eigenfunctions, and if $\lambda_n^+ \ne \lambda_n^-$ then the solutions are distinct. Based on the eigenvalues for the particular classical cases, we can conclude that both sets of eigenfunctions are linearly independent in the Laguerre and Hermite cases ($\sigma''=0$) as:
\begin{equation}
\lambda_n^- < \lambda_{n-1}^- < \cdots < \lambda_0^- < 0 < \lambda_0^+ < \cdots < \lambda_{n-1}^+ < \lambda_n^+,
\end{equation}
but Jacobi polynomials require further investigation. If we let $\gamma = \alpha+\beta+2c-1$, then a nice formula for the eigenvalues is:
\begin{equation}\label{eq:AssociatedJacobiSpectra}
\lambda_n^\pm = (n+1)(n\pm\gamma+1),
\end{equation}
and it follows that $\lambda_n^- = \lambda_{n-\gamma}^+$. In contrast to the Laguerre and Hermite cases, $\lambda_n^-$ are asymptotically positive and if $\gamma\in\Z$, then there are countably infinite equal eigenvalues among the two families. The condition $\gamma\in\Z$ corresponds to a set of lines in the $\alpha\beta$-Jacobi parameter plane. In particular, if $\gamma = 0$, then the linear independence of both families degenerates as eigenvalues of the polynomial solutions of the {\em same} degree are equal. This occurs only if $c=1$ and $\alpha+\beta=-1$.

The orthogonal polynomial connection problem between families $\{p_n(x)\}_{n=0}^\infty$ and $\{q_\ell(x)\}_{\ell=0}^\infty$ is to find connection coefficients organized in an upper-triangular matrix $V$ such that:
\begin{equation}
p_n(x) = \sum_{\ell=0}^n V_{\ell,n} q_\ell(x).
\end{equation}
Fast orthogonal polynomial transforms have a rich history~\cite{Keiner-11}. Notable methods include: the use of asymptotics with rigorous error bounds to related synthesis and analysis to discrete sine and cosine transforms~\cite{Orszag-13-86,Mori-Suda-Sugihara-40-3612-99,Hale-Townsend-36-A148-14,Slevinsky-38-102-18}; hierarchical and fast multipole methods (FMMs)~\cite{Greengard-Rokhlin-73-325-87,Alpert-Rokhlin-12-158-91,Keiner-31-2151-09}; Hadamard matrix factorizations~\cite{Townsend-Webb-Olver-87-1913-18}; and, divide-and-conquer methods for structured eigenvalue problems~\cite{Keiner-30-26-08,Olver-Slevinsky-Townsend-29-573-20}.

To develop fast algorithms for the associated classical -- classical connection problem, it will be important to review the rather general method of Olver, Slevinsky, and Townsend~\cite{Olver-Slevinsky-Townsend-29-573-20} for the classical cases. By using multiplication by $x$, differentiation, raising, and lowering operators for classical orthogonal polynomials, it is shown that Eq.~\eqref{eq:COPEIG} can be expressed as:
\begin{align}
\left(\sigma\DD^2+\tau\DD\right)\begin{pmatrix} p_0 & p_1 & \cdots\end{pmatrix} & = \begin{pmatrix} p_0 & p_1 & \cdots\end{pmatrix}\Lambda,\\
\left(\sigma\DD^2+\tau\DD\right)\begin{pmatrix} q_0 & q_1 & \cdots\end{pmatrix}V & = \begin{pmatrix} q_0 & q_1 & \cdots\end{pmatrix}V\Lambda,\\
\begin{pmatrix} r_0 & r_1 & \cdots\end{pmatrix}AV & = \begin{pmatrix} r_0 & r_1 & \cdots\end{pmatrix}BV\Lambda,\label{eq:CCP}
\end{align}
where $\{r_n(x)\}_{n=0}^\infty$ is yet another classical family and $A$ and $B$ are upper-triangular and banded matrices. Dropping the polynomials on both sides of Eq.~\eqref{eq:CCP}, we obtain an upper-triangular and banded generalized eigenvalue problem for the connection coefficients, $AV = BV\Lambda$.

For example, describing the Jacobi--Jacobi connection problem in terms of the matrices in Appendix~\ref{Appendix:JacobiRecurrences}, we find:
\begin{align}
& \left(\sigma\DD^2+\tau^{(\alpha,\beta)}\DD\right)\begin{pmatrix} P_0^{(\alpha,\beta)} & P_1^{(\alpha,\beta)} & \cdots\end{pmatrix} = \begin{pmatrix} P_0^{(\alpha,\beta)} & P_1^{(\alpha,\beta)} & \cdots\end{pmatrix}\Lambda,\\
& \left(\sigma\DD^2+\tau^{(\alpha,\beta)}\DD\right)\begin{pmatrix} P_0^{(\gamma,\delta)} & P_1^{(\gamma,\delta)} & \cdots\end{pmatrix}V = \begin{pmatrix} P_0^{(\gamma,\delta)} & P_1^{(\gamma,\delta)} & \cdots\end{pmatrix}V\Lambda,\\
& \begin{pmatrix} P_0^{(\gamma+1,\delta+1)} & P_1^{(\gamma+1,\delta+1)} & \cdots\end{pmatrix}\left[-L_{(\gamma+2,\delta+2)}^{(\gamma+1,\delta+1)} D_{(\gamma,\delta)}^{(\gamma+2,\delta+2)}  + \tau^{(\alpha,\beta)}(M_{(\gamma+1,\delta+1)}) D_{(\gamma,\delta)}^{(\gamma+1,\delta+1)}\right]V\nonumber\\
& = \begin{pmatrix} P_0^{(\gamma+1,\delta+1)} & P_1^{(\gamma+1,\delta+1)} & \cdots\end{pmatrix}R_{(\gamma,\delta)}^{(\gamma+1,\delta+1)}V\Lambda.
\end{align}
Appendix~\ref{Appendix:LaguerreRecurrences} compiles the analogous identities for Laguerre polynomials.

The connection coefficients convert not only the polynomials but also their expansions via:
\begin{equation}
\begin{pmatrix} p_0 & p_1 & \cdots\end{pmatrix}\begin{pmatrix} c_0\\c_1\\\vdots\end{pmatrix} = \begin{pmatrix} q_0 & q_1 & \cdots\end{pmatrix}V\begin{pmatrix} c_0\\c_1\\\vdots\end{pmatrix} = \begin{pmatrix} q_0 & q_1 & \cdots\end{pmatrix}\begin{pmatrix} d_0\\d_1\\\vdots\end{pmatrix},
\end{equation}
and the implied matrix-vector product $\bs{d} = V\bs{c}$.

We will extend this technique to any of the associated classical differential equations~\eqref{eq:ACOPEIG1}\,--\,\eqref{eq:ACOPEIG2a} to find structured forms for the associated classical -- classical connection problem.

\section{Fast classical orthogonal polynomial transforms}

We have shown that the classical orthogonal polynomial connection coefficients are generalized eigenvectors of upper-triangular and banded eigenproblems of the form:
\begin{equation}\label{eq:GEP}
AV = BV\Lambda.
\end{equation}
It is clear that the generalized eigenvalues are the ratios of the diagonal entries of the matrices, $\Lambda_{i,i} = A_{i,i}/B_{i,i}$, and we may approach the generalized eigenvectors by dividing and conquering. 

In what follows, we consider the $n \times n$ truncation of this system. That is, we consider $A, B \in \R^{n\times n}$.  Let $s = \lfloor\frac{n}{2}\rfloor$ and let $b$ be the upper bandwidth of both $A$ and $B$. We suppose that $V$ has the form:
\begin{equation}\label{eq:factoredV}
V = \begin{pmatrix} V_{1,1}\\ & V_{2,2}\end{pmatrix}\begin{pmatrix} I & V_{1,2}\\ & I\end{pmatrix},
\end{equation}
where $V_{1,1}\in\R^{s\times s}$, $V_{1,2}\in\R^{s\times(n-s)}$, and $V_{2,2}\in\R^{(n-s)\times(n-s)}$ are all roughly the same size. Block dividing $A$, $B$, and $\Lambda$ conformably, we have:
\[
\begin{pmatrix} 
A_{1,1} & A_{1,2}\\ & A_{2,2}\end{pmatrix}\begin{pmatrix} V_{1,1} & V_{1,1}V_{1,2}\\ 
& V_{2,2}\end{pmatrix} = \begin{pmatrix} B_{1,1} & B_{1,2}\\ 
& B_{2,2}\end{pmatrix}\begin{pmatrix} V_{1,1}\Lambda_1 & V_{1,1}V_{1,2}\Lambda_2\\ 
& V_{2,2}\Lambda_2\end{pmatrix}.
\]
This decomposes into the two diagonal subproblems $A_{i,i}V_{i,i} = B_{i,i}V_{i,i}\Lambda_i$, for $i=1,2$ and the off-diagonal problem:
\[
A_{1,1}V_{1,1}V_{1,2} + A_{1,2}V_{2,2} = B_{1,1}V_{1,1}V_{1,2}\Lambda_2 + B_{1,2}V_{2,2}\Lambda_2,
\]
a two-term matrix equation for the block $V_{1,2}$.

Since $V_{1,1}^{-1} B_{1,1}^{-1} A_{1,1}V_{1,1} = \Lambda_1$, the matrix equation can be cast into the following form: 
\[
\Lambda_1V_{1,2} - V_{1,2}\Lambda_2 = V_{1,1}^{-1}B_{1,1}^{-1}\left(B_{1,2}V_{2,2}\Lambda_2 - A_{1,2}V_{2,2}\right).
\]
Since $A$ and $B$ are banded with an upper bandwidth of $b$, it follows that $A_{1,2}$ and $B_{1,2}$ are matrices with only a small number of nonzero entries, arranged in a lower-triangular fashion in their bottom left corners. Thus, the matrix-matrix product $A_{1,2}V_{2,2}$ results in a rank-$b$ matrix with nonzero entries only in the last $b$ rows. Similarly, due to the zero pattern in the matrix $A_{1,2}V_{2,2} - B_{1,2}V_{2,2}\Lambda_2$ only the last $b$ columns of $B_{1,1}^{-1}$ are needed, which are calculated with back substitution using the last $b$ columns of the $s\times s$ identity. We represent this block as the outer product of an $s\times b$ matrix $X$ and a $b\times (n-s)$ matrix $Y$:
\[
V_{1,1}^{-1}B_{1,1}^{-1}\left(A_{1,2}V_{2,2} - B_{1,2}V_{2,2}\Lambda_2\right) =: XY^\top.
\]
To conquer this eigenproblem, we must then solve the diagonal Sylvester matrix equation:
\begin{equation}\label{eq:diagonalmatrixequation}
\Lambda_1V_{1,2} - V_{1,2}\Lambda_2 = -XY^\top.
\end{equation}
In~\cite{Grasedyck-13-01,Olver-Slevinsky-Townsend-29-573-20}, the solution is examined component-wise where it is found to be a Hadamard product of the rank-$b$ matrix $-XY^\top$ and a Cauchy matrix, which is approximated hierarchically:
\begin{equation}\label{eq:V12componentwise}
(V_{1,2})_{\ell,n} = \dfrac{(-XY^\top)_{\ell,n}}{(\Lambda_1)_{\ell,\ell} - (\Lambda_2)_{n,n}}.
\end{equation}
Due to the separation of spectra in $\Lambda_1$ and $\Lambda_2$, this hierarchical approach offers an approximate matrix-vector product involving $V_{1,2}$ in $\OO(bn\log n\log(1/\epsilon))$ flops and thereby $V$ in $\OO(bn\log^2n\log(1/\epsilon))$ flops. By the FMM, the approximate matrix-vector product with $V_{1,2}$ may be further accelerated to $\OO(bn\log(1/\epsilon))$ flops and thereby $V$ in $\OO(bn\log n\log(1/\epsilon))$ flops. As $(\Lambda_1)_{\ell,\ell} < (\Lambda_2)_{n,n}$ for all $\ell$ and $n$, the integral representation:
\begin{equation}\label{eq:V12integralrepresentation}
V_{1,2} = \int_0^\infty \exp(t\Lambda_1) XY^\top \exp(-t\Lambda_2)\ud t,
\end{equation}
may be approximated via quadrature rules in order to accelerate the linear algebra~\cite{Grasedyck-13-01}. Eq.~\eqref{eq:V12integralrepresentation} is still valid if $\Lambda_1$ and $\Lambda_2$ are not diagonal matrices, but the condition on the spectra remains.

Diagonal matrix equations such as Eq.~\eqref{eq:diagonalmatrixequation} may also be approximated by the factored alternating direction implicit (fADI) method~\cite{Peaceman-Rachford-3-28-55,Wachspress-66,Benner-Li-Truhar-233-1035-09}. For normal Sylvester matrix equations, the precise number of shifts and their coordinates are computable values related to the Jacobian elliptic function and elliptic integral solution of Zolotarev's third problem~\cite{Zolotarev-30-1-77,Beckermann-Townsend-61-319-19}, that of minimizing the maximum absolute value of a rational function on one set divided by its minimum absolute value on another --- sets which enclose the disjoint spectra of $\Lambda_1$ and $\Lambda_2$. For nonnormal Sylvester matrix equations, the relationship between the shifts and the error after $J$ steps of fADI is more complicated.

The structured form of the generalized eigenvectors permits inversion and transposition with the same computational complexity. For example:
\[
\begin{aligned}
V^{-1} = \begin{pmatrix} I & -V_{1,2}\\ & I\end{pmatrix}\begin{pmatrix} V_{1,1}^{-1}\\ & V_{2,2}^{-1}\end{pmatrix},\\
V^\top = \begin{pmatrix} I\\ V_{1,2}^\top & I\end{pmatrix}\begin{pmatrix} V_{1,1}^\top\\ & V_{2,2}^\top\end{pmatrix},\\
V^{-\top} = \begin{pmatrix} V_{1,1}^{-\top}\\ & V_{2,2}^{-\top}\end{pmatrix} \begin{pmatrix} I\\ -V_{1,2}^\top & I\end{pmatrix}.
\end{aligned}
\]

With the structured forms for the eigenvectors, any submultiplicative operator norm $\norm{\cdot}$ can be estimated recursively via:
\begin{align*}
\norm{V} & \le \norm{\begin{pmatrix} V_{1,1}\\ & V_{2,2}\end{pmatrix}} \norm{\begin{pmatrix} I & V_{1,2}\\ & I\end{pmatrix}},\\
& \le \max\{\norm{V_{1,1}}, \norm{V_{2,2}}\} \left( \norm{I_n} + \norm{V_{1,2}}\right).
\end{align*}
Estimates on the induced $2$-norms of the leaves follow from H\"older's inequality and by comparison with the Frobenius norm so that singular values need not be computed~\cite{Golub-Van-Loan-13}. Estimates for the norms of the inverse transforms (and thereby condition number estimates) also follow naturally from the recursive argument above.

\section{Fast associated classical orthogonal polynomial transforms}

We consider the quadratic eigenvalue problem in Eq.~\eqref{eq:ACOPEIG2a}. Expanding both $\{p_n^\pm(x;c)\}_{n=0}^\infty$ in the same classical orthogonal polynomial basis, say $\{q_\ell(x)\}_{\ell=0}^\infty$, the resulting discretizations are matrix equations of the form:
\begin{align}
AV^- + BV^-\Lambda^- & = CV^-(\Lambda^-)^2,\quad{\rm and}\label{eq:NQEP}\\
AV^+ + BV^+\Lambda^+ & = CV^+(\Lambda^+)^2,\label{eq:PQEP}
\end{align}
where $A,B,C$ are upper-triangular and banded matrices, $V^\pm$ are the upper-triangular connection coefficients, and $\Lambda^\pm$ are the respective diagonal quadratic eigenvalue matrices. For example, with Jacobi polynomials, the same procedure we use for the classical transforms results in:
\begin{align}
A & = -L_{(\gamma+4,\delta+4)}^{(\gamma+2,\delta+2)}D_{(\gamma,\delta)}^{(\gamma+4,\delta+4)} + 5\sigma'(M_{(\gamma+2,\delta+2)})L_{(\gamma+3,\delta+3)}^{(\gamma+2,\delta+2)}D_{(\gamma,\delta)}^{(\gamma+3,\delta+3)}\nonumber\\
& \quad + \left[\tau^2+2\tau'\sigma-2\tau\sigma'-6\sigma\sigma''+4\lambda_{c-1}\sigma - 3\sigma'^2\right](M_{(\gamma+2,\delta+2)})D_{(\gamma,\delta)}^{(\gamma+2,\delta+2)}\\
& \quad + 3\left[\tau\tau' + 2\lambda_{c-1}\sigma' - (\tau+\sigma')\sigma''\right](M_{(\gamma+2,\delta+2)})R_{(\gamma+1,\delta+1)}^{(\gamma+2,\delta+2)}D_{(\gamma,\delta)}^{(\gamma+1,\delta+1)} + \left[2\lambda_{c-1}\sigma'' - \tau'(\sigma''-\tau')\right]R_{(\gamma,\delta)}^{(\gamma+2,\delta+2)},\nonumber\\
B & = 2\sigma(M_{(\gamma+2,\delta+2)})D_{(\gamma,\delta)}^{(\gamma+2,\delta+2)} + 3\sigma'(M_{(\gamma+2,\delta+2)})R_{(\gamma+1,\delta+1)}^{(\gamma+2,\delta+2)}D_{(\gamma,\delta)}^{(\gamma+1,\delta+1)} + \sigma''R_{(\gamma,\delta)}^{(\gamma+2,\delta+2)},\\
C & = R_{(\gamma,\delta)}^{(\gamma+2,\delta+2)}.
\end{align}

Letting $W^\pm = V^\pm\Lambda^\pm$, the quadratic eigenvalue problem may be linearized, collecting both families of eigenvectors in the same $2\times2$ block equation:
\begin{equation}
\begin{pmatrix} A & B\\ & I\end{pmatrix}\begin{pmatrix} V^- & V^+\\ W^- & W^+\end{pmatrix} = \begin{pmatrix} & C\\ I\end{pmatrix} \begin{pmatrix} V^- & V^+\\ W^- & W^+\end{pmatrix} \begin{pmatrix} \Lambda^-\\ & \Lambda^+\end{pmatrix}.
\end{equation}

Any $2\times2$ block equation with upper-triangular and banded structure may be permuted to an upper-triangular and banded equation involving $2\times2$ blocks. This permutation, known as the perfect shuffle in reference to a deck of cards, can be described as taking the odd columns before the even columns of the $2n\times 2n$ identity, $P = I_{1:2n, [1:2:2n, 2:2:2n]}$. Then:
\[
\underbrace{P\begin{pmatrix} A & B\\ & I\end{pmatrix}P^\top}_{=:\bs{A}} \underbrace{P\begin{pmatrix} V^- & V^+\\ W^- & W^+\end{pmatrix}P^\top}_{=:\bs{V}} = \underbrace{P\begin{pmatrix} & C\\ I\end{pmatrix}P^\top}_{=:\bs{B}} \underbrace{P\begin{pmatrix} V^- & V^+\\ W^- & W^+\end{pmatrix}P^\top}_{=\bs{V}} \underbrace{P\begin{pmatrix} \Lambda^-\\ & \Lambda^+\end{pmatrix}P^\top}_{=:\bs{\Lambda}},
\]
is in fact an upper-triangular and banded generalized eigenvalue problem with $2\times2$ blocks:
\begin{equation}\label{eq:2x2blockGEP}
\bs{A} \bs{V} = \bs{B} \bs{V} \bs{\Lambda}.
\end{equation}
Eq.~\eqref{eq:2x2blockGEP} is almost in the form of Eq.~\eqref{eq:GEP}: were we to relate the former to the latter, we would enable the divide-and-conquer strategy that already accelerates the classical orthogonal polynomial connection problem. Since $\bs{V}$ is {\em almost} upper-triangular, we solve the $n$ generalized eigenvalue problems of size $2\times2$ on the main diagonal and use the solutions to produce a sequence of Givens rotations~\cite{Golub-Van-Loan-13} to upper-triangularize $\bs{V}$. If:
\[
\bs{V} = Q_V R_V,
\]
and if:
\[
\bs{B} Q_V = Q_B R_B,
\]
then it follows that $Q_B^\top \bs{A} Q_V$ is upper-triangular and banded. That is, the problem:
\[
\left(Q_B^\top \bs{A} Q_V\right) R_V = R_B R_V \bs{\Lambda},
\]
is an upper-triangular and banded generalized eigenvalue problem of size $2n\times 2n$. The divide-and-conquer approach can be used to approximate $R_V$, which together with $Q_V$, provides a structured form for the generalized eigenvectors. Calculating $Q_V$ costs $\OO(n)$ flops, while $\bs{B}Q_V$, $Q_B$, $R_B$, and $Q_B^\top \bs{A} Q_V$ cost $\OO(nb)$ flops, where $b$ is the upper bandwidth of both $\bs{A}$ and $\bs{B}$.

The classical orthogonal polynomials all have well-separated spectra. From Theorem~\ref{theorem:ACOPQEP} and considering the associated Jacobi polynomial spectra in Eq.~\eqref{eq:AssociatedJacobiSpectra}, if $\gamma$ is an integer, there are eigenspaces with algebraic multiplicity greater than one. For example, for first associated Legendre polynomials, $(\alpha,\beta,c)=(0,0,1)$ and:
\[
\bs{\Lambda} = \begin{pmatrix} \lambda_0^-\\ & \lambda_0^+\\ & & \lambda_1^-\\ & & & \lambda_1^+\\ & & & & \lambda_2^-\\ & & & & & \lambda_2^+\\ & & & & & & \ddots\end{pmatrix} = \left(\begin{array}{cccc|ccc} 0 & & & & & &\\ & 2 & & & &\\ & & 2 & & & &\\ & & & 6 & & &\\\hline & & & & 6\\ & & & & & 12\\ & & & & & & \ddots \end{array}\right).
\]
We have added the partitioning lines to help illustrate the problem that arises in the divide-and-conquer approach when eigenvalues are semi-simple: the component-wise formula for $V_{1,2}$ in Eq.~\eqref{eq:V12componentwise} is invalid at any indices $\ell$ and $n$ such that $(\Lambda_1)_{\ell,\ell} = (\Lambda_2)_{n,n}$. Instead, given a tolerance $\epsilon>0$, we define $\hat{V}_{1,2}$ by:
\[
(\hat{V}_{1,2})_{\ell,n} = \left\{\begin{array}{cc} \dfrac{(-XY^\top)_{\ell,n}}{(\Lambda_1)_{\ell,\ell} - (\Lambda_2)_{n,n}} & {\rm if}~~\abs{(\Lambda_1)_{\ell,\ell}-(\Lambda_2)_{n,n}} > \epsilon \max\{\abs{(\Lambda_1)_{\ell,\ell}}, \abs{(\Lambda_2)_{n,n}}\},\\ 0 & {\rm otherwise},\end{array}\right.
\]
so that $V_{1,2} = \hat{V}_{1,2} + S_{1,2}$, where $S_{1,2}$ is a sparse matrix with $\OO(\gamma)$ nonzero entries. We find the nonzero entries in $S_{1,2}$ by comparing $V_{1,1} V_{1,2}$ to ``true'' eigenvectors found by shifting-and-inverting in $\OO(n)$ flops per eigenvector. If $\gamma=\alpha+\beta+2c-1=0$ so that $\lambda_n^- \equiv \lambda_n^+$, the linearization of the quadratic eigenvalue problem degenerates, a situation that occurs only if $c=1$ along the line $\alpha+\beta=-1$. We provide a solution for this problem in \S~\ref{subsection:firstorderproblem}.

\begin{figure}[htbp]
\begin{center}
\begin{tabular}{cc}
\hspace*{-0.2cm}\includegraphics[width=0.53\textwidth]{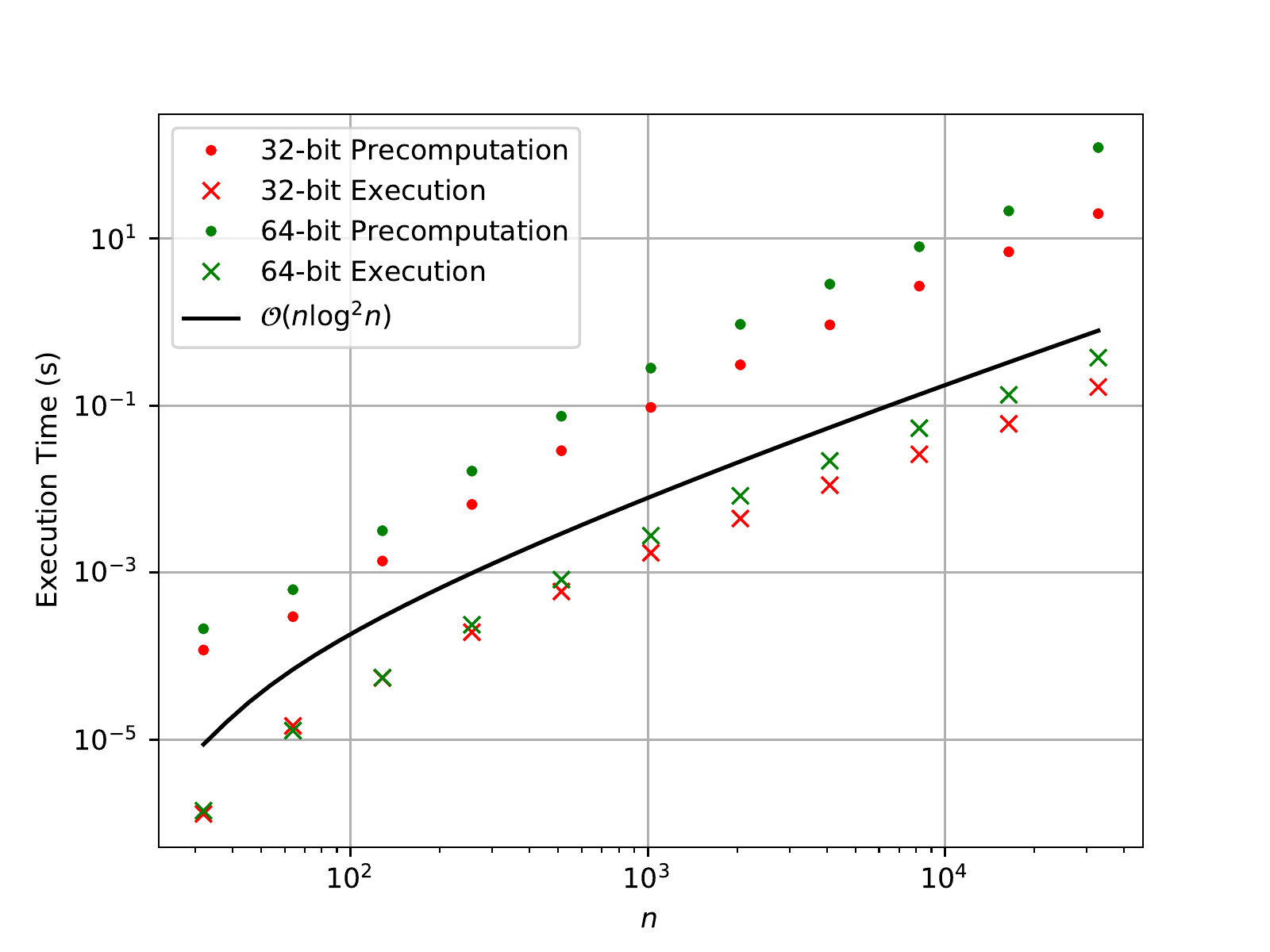}&
\hspace*{-0.65cm}\includegraphics[width=0.53\textwidth]{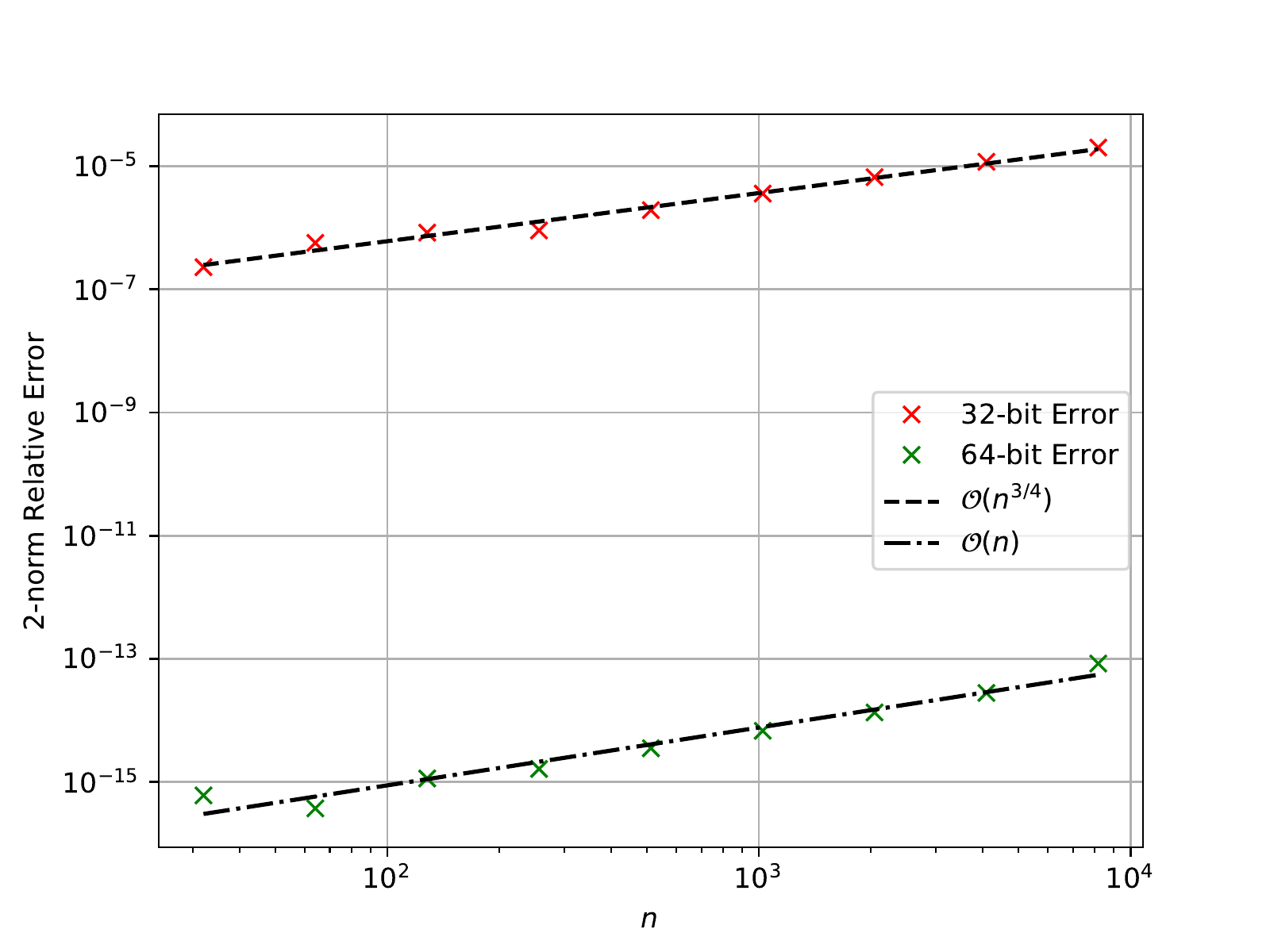}\\
\end{tabular}
\caption{Left: Computational timings for a fast approximate matrix-vector product with first associated Legendre--Legendre connection coefficients. Right: Error growth given first associated Legendre coefficients $c_k = (k+1)^{-1}$ for $k=0,\ldots,n-1$.}
\label{fig:errorandtimings}
\end{center}
\end{figure}

Figure~\ref{fig:errorandtimings} shows timings for the precomputation of the structured form of $\bs{V}$ and the execution of a matrix-vector product with a set of first associated Legendre coefficients with modest decay in single and double precision. The timings scale with $\OO(n\log^2n)$ because our implementation~\cite{Slevinsky-GitHub-FastTransformsC} uses a hierarchical approximation of the Cauchy matrices; it may be improved to $\OO(n\log n)$ by an implementation of the FMM. The figure also shows the low algebraic degree-dependent scaling of the $2$-norm relative error.

It has been observed that the accuracy of the divide-and-conquer approach depends on the conditioning of the eigenvectors themselves~\cite{Keiner-30-26-08,Olver-Slevinsky-Townsend-29-573-20}. For classical orthogonal polynomial connection problems within the same family, the $2$-norm condition number may be related to parameter differences: for ultraspherical--ultraspherical connection problems, the conditioning is related to an algebraic power of $n$ on the order of the absolute difference of the ultraspherical parameters, say  $\OO(n^{\abs{\lambda-\mu}})$. This generalizes to the Jacobi and Laguerre families as well. Hence, while it is theoretically possible to accelerate an out-of-family connection problem such as Hermite to Legendre, the conditioning of the problem may be a severe numerical impediment. Generally speaking, the ill-conditioning of a connection problem may be reduced (but not eliminated) by orthonormalization.

For the associated classical -- classical connection problem, the coupling of the positive and negative eigenspaces is a novel cause for concern. Heuristically, Eq.~\eqref{eq:COPEIG} promotes nonuniform oscillatory behaviour in the polynomial solutions on $D=\supp(\mu)$. Eq.~\eqref{eq:ACOPEIG2a} promotes a similar oscillatory behaviour for the positive solutions $p_n^+(x;c)$, but for solutions with a negative eigenvalue, the polynomials $p_n^-(x;c)$ are permitted to develop an oscillation-free boundary layer. Such solutions are inconsistent with any classical orthogonal polynomial family to which the associated polynomials are connected. This mismatch between the behaviours of $p_n^-(x;c)$ and $q_\ell(x)$ may be the cause for a catastrophic growth in the conditioning of the coupled connection problem of Eq.~\eqref{eq:2x2blockGEP}, eigenvalue separation notwithstanding. Figure~\ref{fig:conditionnumber} demonstrates that the normalized first associated connection problems are all reasonably well-conditioned, but that the condition numbers of the coupled connection problem of Eq.~\eqref{eq:2x2blockGEP} grows too rapidly to be useful in the Hermite, Laguerre, and high-parameter Jacobi connection problems. We believe the (block) triangular structure justifies the accuracy in the numerical estimates on the condition numbers even as they may grow beyond $1/\epsilon$ to astronomical figures.

\begin{figure}[htbp]
\begin{center}
\begin{tabular}{cc}
\hspace*{-0.2cm}\includegraphics[width=0.53\textwidth]{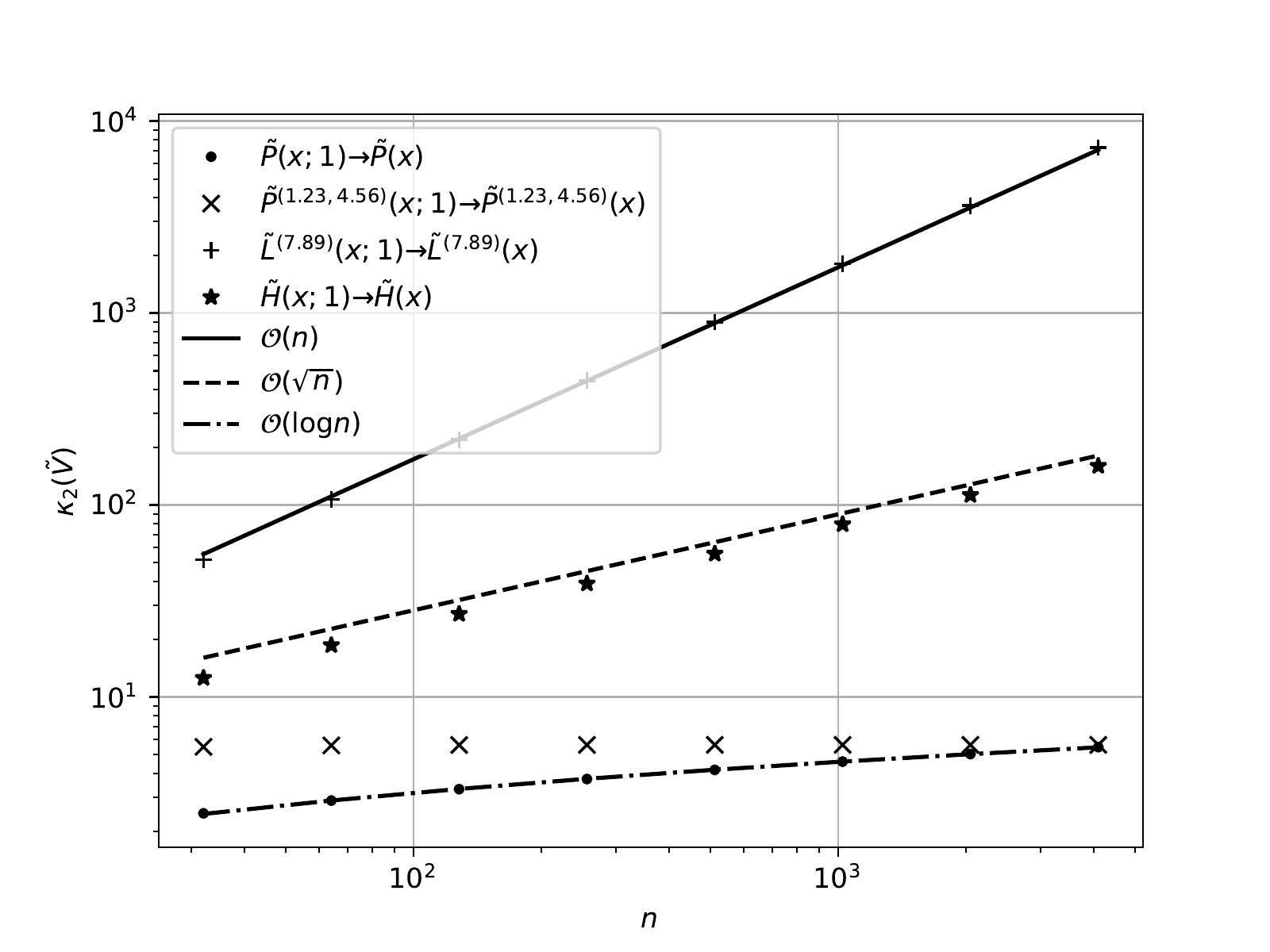}&
\hspace*{-0.65cm}\includegraphics[width=0.53\textwidth]{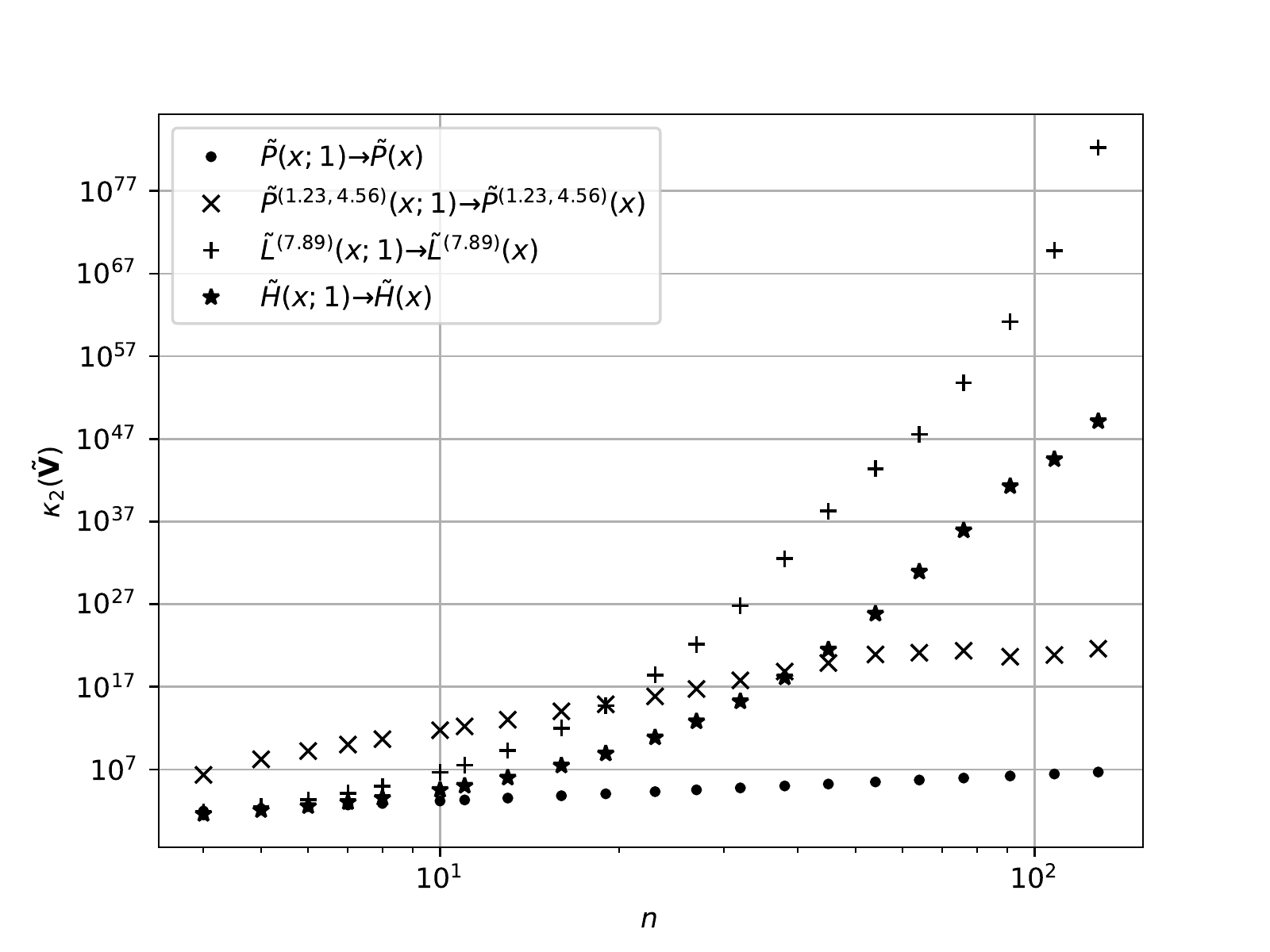}\\
\end{tabular}
\caption{Left: The degree-dependence of $\kappa_2(\tilde{V})$ for first associated Legendre, Jacobi, Laguerre, and Hermite connection problems. Right: The degree dependence of $\kappa_2(\bs{\tilde{V}})$ for the same connection problems. In both plots, a tilde overtop the polynomials indicates orthonormalization, related to the standard normalization via $p_n(x;c) = \sqrt{h_{n+c}}\tilde{p}_n(x;c)$, where $h_n$ is defined in Table~\ref{table:COP}.}
\label{fig:conditionnumber}
\end{center}
\end{figure}

Whereas the formulation of the associated classical -- classical connection problem as a linearized quadratic eigenvalue problem formally solves our problem, the condition number of specific problems prohibits the method from delivering any reasonable results. For these reasons, we describe starting points for two alternative methods in the next two subsections to provide avenues for future research.

\subsection{Uncoupling quadratic eigenvectors}

The divide-and-conquer strategy may be applied directly to either of Eqs.~\eqref{eq:NQEP} and~\eqref{eq:PQEP} without coupling the solutions. Both solutions of the quadratic eigenvalue problem are of the form:
\[
AV+BV\Lambda = CV\Lambda^2.
\]
We suppose that $V$ has the factored form in Eq.~\eqref{eq:factoredV}. Block dividing all matrices conformably:
\[
\begin{pmatrix} A_{1,1} & A_{1,2}\\ & A_{2,2}\end{pmatrix}\begin{pmatrix} V_{1,1} & V_{1,1}V_{1,2}\\ 
& V_{2,2}\end{pmatrix} + \begin{pmatrix} B_{1,1} & B_{1,2}\\ 
& B_{2,2}\end{pmatrix}\begin{pmatrix} V_{1,1}\Lambda_1 & V_{1,1}V_{1,2}\Lambda_2\\ 
& V_{2,2}\Lambda_2\end{pmatrix} = \begin{pmatrix} C_{1,1} & C_{1,2}\\ 
& C_{2,2}\end{pmatrix}\begin{pmatrix} V_{1,1}\Lambda_1^2 & V_{1,1}V_{1,2}\Lambda_2^2\\ 
& V_{2,2}\Lambda_2^2 \end{pmatrix}.
\]
As the diagonal blocks have the same structure as the original problem, we assume they have been solved before seeking a solution to the off-diagonal block $V_{1,2}$, that now satisfies:
\[
A_{1,1}V_{1,1}V_{1,2} + A_{1,2}V_{2,2} + B_{1,1}V_{1,1}V_{1,2}\Lambda_2 + B_{1,2}V_{2,2}\Lambda_2 = C_{1,1}V_{1,1}V_{1,2}\Lambda_2^2 + C_{1,2}V_{2,2}\Lambda_2^2.
\]
Since $V_{1,1}^{-1} C_{1,1}^{-1} A_{1,1}V_{1,1} = \Lambda_1^2 - V_{1,1}^{-1} C_{1,1}^{-1} B_{1,1}V_{1,1}\Lambda_1$, the matrix equation can be cast in the following form:
\begin{align}
\Lambda_1^2 V_{1,2} - V_{1,2}\Lambda_2^2 - V_{1,1}^{-1}C_{1,1}^{-1} B_{1,1} V_{1,1}\left(\Lambda_1 V_{1,2} - V_{1,2}\Lambda_2\right) & = V_{1,1}^{-1}C_{1,1}^{-1}\left(C_{1,2}V_{2,2}\Lambda_2^2 - A_{1,2}V_{2,2} - B_{1,2}V_{2,2}\Lambda_2\right),\nonumber\\
& =: -XY^\top.\label{eq:AssociatedBlockProblem4}
\end{align}
As in the classical divide-and-conquer scheme, $\rank(XY^\top)$ depends only on the bandwidths of $A$, $B$, and $C$.

So far, Eq.~\eqref{eq:AssociatedBlockProblem4} does not utilize the quadratic aspect of the eigenvalue problem, and we would have obtained an apparently-similar equation for a two-parameter eigenvalue problem with known parameters, say $\Gamma$ and $\Omega$, had we started with Eq.~\eqref{eq:ACOPEIG1}. However, the quadratic character of the eigenvalue problem leads to a considerable simplification.

\begin{lemma}\label{lemma:KroneckerFactors}
For any two square matrices $A$ and $B$ and a conformable matrix $Z$:
\begin{equation}
A^2Z-ZB^2 = A\left(AZ-ZB\right) + \left(AZ-ZB\right)B.
\end{equation}
\end{lemma}

Lemma~\ref{lemma:KroneckerFactors} allows us to solve Eq.~\eqref{eq:AssociatedBlockProblem4} by factoring $\Lambda_1^2 V_{1,2} - V_{1,2}\Lambda_2^2$ and using the following two step procedure. First, we would solve for $W$ in:
\begin{equation}\label{eq:ACOPW}
\left(\Lambda_1 - V_{1,1}^{-1}C_{1,1}^{-1} B_{1,1} V_{1,1}\right)W + W\Lambda_2 = -XY^\top.
\end{equation}
Then, we would solve for $V_{1,2}$ in:
\begin{equation}\label{eq:ACOPV12}
\Lambda_1V_{1,2} - V_{1,2}\Lambda_2 = W.
\end{equation}

Although this scheme conserves the problem size and uncouples positive from negative eigenvectors, it is not our scheme of preference as Eq.~\eqref{eq:ACOPW} is a nonnormal matrix equation for $W$. Although it is likely that there exists a low-rank solution, none of the aforementioned algorithms provides computable bounds on the rank of $W$: the analogy to a Hadamard product with a Cauchy matrix is lost when one of the terms in the matrix equation is not diagonal; the (optimal) choice of shifts in fADI must either take into account any ill-conditioning in Eq.~\eqref{eq:ACOPW} that results from diagonalizing $\Lambda_1 - V_{1,1}^{-1}C_{1,1}^{-1} B_{1,1} V_{1,1}$ or must be chosen according to its field of values; and, the quadrature-based approach based on the semi-infinite integral either requires the fast approximate action of a nonnormal matrix exponential or an estimate of the departure from normality for the quadrature error. Theoretical properties on nonnormal matrix equations are sparse in the literature. Notably, Baker, Embree, and Sabino~\cite{Baker-Embree-Sabino-36-656-15} discuss specific structured nonnormal matrix equations which are found to admit low-rank solutions despite (and even aided by) the nonnormality.

Algorithmic considerations aside, it is important for invertibility to show that the spectra in the two terms in Eqs.~\eqref{eq:ACOPW} and~\eqref{eq:ACOPV12} are disjoint. For Eq.~\eqref{eq:ACOPV12}, this follows from Theorem~\ref{theorem:ACOPQEP}. For Eq.~\eqref{eq:ACOPW}, the disjointedness of the spectra requires more attention. Since $V_{1,1}$, $B_{1,1}$, and $C_{1,1}$ are upper-triangular and $\Lambda_1$ is diagonal:
\begin{equation}
\sigma(\Lambda_1 - V_{1,1}^{-1}C_{1,1}^{-1} B_{1,1} V_{1,1}) = \sigma(\Lambda_1 - C_{1,1}^{-1}B_{1,1}).
\end{equation}
By Eq.~\eqref{eq:ACOPEIG2a}, it follows that the spectrum of $C_{1,1}^{-1}B_{1,1}$ is equal to the set of the first $s$ eigenvalues of $\BB$. The three classical families behave differently.

\begin{enumerate}
\item For Jacobi polynomials, $\sigma=x^2-1$, $\sigma'=2x$, and $\sigma''=2$ so that:
\[
2\sigma\DD^2+3\sigma'\DD+\sigma''\II = 2\left(\sigma\DD^2+\tau^{(\frac{1}{2},\frac{1}{2})}\DD+\II\right),
\]
hence $\sigma(C_{1,1}^{-1}B_{1,1}) = \{2\lambda_\nu^{(\frac{1}{2},\frac{1}{2})}+2\}_{\nu=0}^{s-1}$, where $\lambda_n^{(\frac{1}{2},\frac{1}{2})}$ are the eigenvalues of the second kind Chebyshev differential equation.
\item For Laguerre polynomials, $\sigma=-x$, $\sigma' = -1$, and $\sigma'' = 0$ so that:
\[
2\sigma\DD^2+3\sigma'\DD+\sigma''\II = -2x\DD^2-3\DD,
\]
hence $\sigma(C_{1,1}^{-1}B_{1,1}) = \{0\}$.
\item Finally, for Hermite polynomials, $\sigma=-1$ and $\sigma'=\sigma''=0$ so that:
\[
2\sigma\DD^2+3\sigma'\DD+\sigma''\II = -2\DD^2,
\]
hence $\sigma(C_{1,1}^{-1}B_{1,1}) = \{0\}$.
\end{enumerate}

For Laguerre and Hermite polynomials, it is clear that $\sigma(\Lambda_1)\cap\sigma(-\Lambda_2) = \emptyset$.

For Jacobi polynomials, a sufficient condition for $\sigma(\Lambda_1-C_{1,1}^{-1}B_{1,1})\cap\sigma(-\Lambda_2) = \emptyset$ is that:
\[
\lambda_{s+c}^{(\alpha,\beta)} + \lambda_{s+c-1}^{(\alpha,\beta)} > 2\left(\lambda_{s-1}^{(\frac{1}{2},\frac{1}{2})}+\lambda_{c-1}^{(\alpha,\beta)}+1\right).
\]
Canceling terms, this reduces to:
\[
(2s+1)(\alpha+\beta+2c) > 0.
\]
Since the order of association $c$ is nontrivially at least one and the Jacobi parameters must satisfy $\alpha,\beta>-1$, there is always a spectral gap\footnote{though that gap may shrink in the limit as $\alpha+\beta\to-2^+$ when $c=1$!}, guaranteeing that the two step procedure defined by Eqs.~\eqref{eq:ACOPW} and~\eqref{eq:ACOPV12} indeed has a solution.

\subsection{The first order of association}\label{subsection:firstorderproblem}

We consider the degree-preserving factored differential equation, Eq.~\eqref{eq:ACOPEIG3a} or equivalently Eq.~\eqref{eq:ACOPEIG3} with $c=1$, where we expand $\{p_n(x;1)\}_{n=0}^\infty$ in the corresponding classical orthogonal polynomial basis, $\{p_\ell(x)\}_{\ell=0}^\infty$. We separate this case for two reasons: firstly, it is reasonable to consider this connection based on the diagonalization of the integral transform in Eq.~\eqref{eq:STSR}; secondly, the form of the matrix equation satisfied by the connection coefficients is simplified considerably.

The latent truth underlying Eq.~\eqref{eq:ACOPEIG3a} is the inhomogeneous second-order problem~\cite{Ronveaux-21-L749-88,Zarzo-Ronveaux-Godoy-49-349-93,Foupouagnigni-Koepf-Ronveaux-162-299-04}:
\begin{equation}\label{eq:ACOPEIG3b}
\left[\sigma\DD^2 + (2\sigma'-\tau)\DD+ (\sigma''-\tau')\right]p_n(x;1) = \lambda_{n+1}p_n(x;1) + \frac{(\sigma''-2\tau')}{A_0}\DD p_{n+1}(x).
\end{equation}
The particular form of the inhomogeneity is relatively easy to discretize: expanding $\{p_n(x;1)\}_{n=0}^\infty$ in $\{p_\ell(x)\}_{\ell=0}^\infty$ and converting this representation to the (also classical) basis of $\{\DD p_{\ell+1}(x)\}_{\ell=0}^\infty$ results in a ``forced'' upper-triangular and banded generalized eigenvalue problem:
\begin{equation}\label{eq:forcedGEP}
AV = BV\Lambda + \Gamma,
\end{equation}
where $\Gamma$ is a diagonal forcing matrix. 

Dividing and conquering, the resulting matrix equation that must be solved is:
\begin{equation}\label{eq:AssociatedBlockProblemc1}
V_{1,1}^{-1}B_{1,1}^{-1} A_{1,1} V_{1,1}V_{1,2} - V_{1,2}\Lambda_2 = V_{1,1}^{-1} B_{1,1}^{-1}\left(B_{1,2}V_{2,2}\Lambda_2 - A_{1,2}V_{2,2}\right) =: -XY^\top.
\end{equation}
Eq.~\eqref{eq:AssociatedBlockProblemc1} is almost the same as Eq.~\eqref{eq:diagonalmatrixequation}, though we can no longer say that $V_{1,1}^{-1}B_{1,1}^{-1} A_{1,1} V_{1,1} = \Lambda_1$.

It is an interesting observation that the differential operator on the left-hand side of Eq.~\eqref{eq:ACOPEIG3b} is the formal adjoint of the classical one. That is, the diagonalization of $B_{1,1}^{-1}A_{1,1}$ in Eq.~\eqref{eq:AssociatedBlockProblemc1} is related to finding the polynomial solutions to the eigenvalue problem:
\begin{equation}\label{eq:innermostprob}
\left[\sigma\DD^2 + (2\sigma'-\tau)\DD+ (\sigma''-\tau')\right]q_n(x) = \DD^2[\sigma q_n] - \DD[\tau q_n] = \omega_n q_n(x).
\end{equation}
\begin{lemma}
The eigenvalues $\omega_n$ of Eq.~\eqref{eq:innermostprob} are given by:
\begin{equation}
\omega_n = \frac{n+1}{2}\left[(n+2)\sigma''-2\tau'\right].
\end{equation}
\end{lemma}
\begin{proof}
The proof follows by comparing coefficients of the monomial $x^n$.
\end{proof}

The implied diagonalization is an issue for all three classical families. For Jacobi polynomials, it is akin to connecting $\{P_n^{(-\alpha,-\beta)}(x)\}_{n=0}^\infty$ to $\{P_\ell^{(\alpha,\beta)}(x)\}_{\ell=0}^\infty$, which is ill-conditioned if $\max\{\alpha, \beta\} \gg 1$. For Hermite and Laguerre polynomials, the differential equations are not related to classical orthogonal polynomial eigenproblems with different parameters.

If $A_{1,1}R_{1,1} = B_{1,1}R_{1,1}\Omega_1$, then letting $Z_{1,2} = R_{1,1}^{-1}V_{1,1}V_{1,2}$ in Eq.~\eqref{eq:AssociatedBlockProblemc1} we find the diagonalized form:
\[
\Omega_1Z_{1,2} - Z_{1,2}\Lambda_2 = -R_{1,1}^{-1}V_{1,1}XY^\top,
\]
and for a solution to this matrix equation to exist, we must show that $\sigma(\Omega_1) \cap \sigma(\Lambda_2) = \emptyset$. A sufficient condition is that:
\begin{align*}
\max\{\omega_0, \omega_{s-1}\} & < \lambda_{s+1},\quad{\rm or}\\
\max\left\{\sigma''-\tau', \frac{s}{2}\left[(s+1)\sigma''-2\tau'\right]\right\} & < \frac{s+1}{2}\left[s\sigma''+2\tau'\right].
\end{align*}
Canceling terms, this is true since $\tau'>0$ for all classical families.

Using the fact that $A_{1,1}V_{1,1} = B_{1,1}V_{1,1}\Lambda_1 + \Gamma_1$, an alternative to solving Eq.~\eqref{eq:AssociatedBlockProblemc1} is to solve the nonnormal matrix equation:
\begin{equation}
\Lambda_1V_{1,2} - V_{1,2}\Lambda_2 + V_{1,1}^{-1}B_{1,1}^{-1}\Gamma_1 V_{1,2} = -XY^\top.
\end{equation}

\begin{remark}
The quadratic eigenvalue problem approach fails when the positive and negative eigenfunctions degenerate, a situation that occurs only if $c=1$ and $\alpha+\beta=-1$. We note that for Jacobi polynomials, $\sigma''-2\tau' = -2(\alpha+\beta+1)$. Therefore, in view of the right-hand side of Eq.~\eqref{eq:ACOPEIG3b}, the forcing term is precisely $0$, which allows us to solve the connection problem between $\{p_n(x;1)\}_{n=0}^\infty$ and in fact any classical orthogonal polynomial basis $\{q_\ell(x)\}_{\ell=0}^\infty$ via an upper-triangular and banded generalized eigenvalue problem.
\end{remark}

\section{Explicit results on the associated classical connection problem}\label{section:ACOPCP}

We collect the known elegant results on the associated classical -- classical orthogonal polynomial connection problem. General formul\ae~for the Jacobi and Laguerre connection coefficients are given by Lewanowicz~\cite{Lewanowicz-65-215-95} as sums of generalized hypergeometric functions. The formul\ae~are likely too complicated to be useful in practice, though we refer the interested reader to the results in case they feel differently. Earlier, Lewanowicz~\cite{Lewanowicz-49-137-93} finds formul\ae~for two special cases in the associated Jacobi--Jacobi connection coefficients, apart from two typographical errors corrected here. The first associated ultraspherical--ultraspherical case is due to Watson~\cite[\S 3.15.2]{Erdelyi-et-al-1-53}; see also Paszkowski~\cite{Paszkowski-136-84}. The first associated Legendre--Legendre case is discovered independently by Temme~\cite[Eq.~(8.30)]{Temme-96}. For the associated Hermite--Hermite problem, the formul\ae~are due to Askey and Wimp~\cite{Askey-Wimp-96-15-84}. For the generalized Laguerre polynomials, $L_n^{(\alpha)}(x;c)$, the special cases of $\alpha=\frac{1}{2}$ in~\cite{Lewanowicz-65-215-95} and also $\alpha=-\frac{1}{2}$ are related to the Hermite problem.

The connection coefficients are conveniently expressed in terms of the gamma function~\cite[\S 5]{Olver-et-al-NIST-10}, defined by:
\[
\Gamma(z) := \int_0^\infty x^{z-1}e^{-x}\ud x, \quad{\rm for}\quad \Re z>0,
\]
and its analytic continuation to $z\in\C\setminus\{-\N_0\}$ by the recurrence $\Gamma(z+1) = z\Gamma(z)$. If one or more happen to be singular in the formul\ae~below, then we take limiting values to preserve continuity.

\begin{lemma}[Lewanowicz~\cite{Lewanowicz-49-137-93}]
The connection coefficients between the associated Jacobi polynomials $\{P_n^{(\alpha,\frac{1}{2})}(x;c)\}_{n=0}^\infty$ and the Jacobi polynomials $\{P_\ell^{(\alpha,\frac{1}{2})}(x)\}_{\ell=0}^\infty$ are given by:
\[
V_{\ell,n} = \left\{\begin{array}{ccc} u_n^{(\alpha, c)} t_{\ell,n}^{(\alpha, c)} h_{\ell,n}^{(\alpha, c)} v_\ell^{(\alpha)} & \for & \ell\le n,\\ 0 & \multicolumn{2}{c}{otherwise.}\end{array}\right.
\]
where:
\begin{align*}
u_n^{(\alpha, c)} & = \dfrac{\Gamma(n+1)\Gamma(n+\frac{3}{2})\Gamma(2n+2\alpha+2c+2)\Gamma(\alpha+c+\frac{3}{2})\Gamma(c+1)}{\Gamma(2n+2)\Gamma(\alpha+2c+\frac{3}{2})\Gamma(n+\alpha+c+\frac{3}{2})\Gamma(n+c+1)},\\
t_{\ell,n}^{(\alpha, c)} & = \dfrac{\Gamma(n-\ell+\frac{1}{2}-\alpha)\Gamma(\alpha+2c+\frac{1}{2})\Gamma(n-\ell+2c)}{\Gamma(\frac{1}{2}-\alpha)\Gamma(n-\ell+\alpha+2c+\frac{1}{2})\Gamma(n-\ell+1)},\\
h_{\ell,n}^{(\alpha, c)} & = \dfrac{\Gamma(n+\ell+2)\Gamma(n+\ell+\alpha+2c+\frac{3}{2})}{\Gamma(n+\ell+\alpha+\frac{5}{2})\Gamma(n+\ell+2\alpha+2c+2)},\\
v_\ell^{(\alpha)} & = (2\ell+\alpha+\tfrac{3}{2})\dfrac{\Gamma(\ell+\alpha+\frac{3}{2})}{\Gamma(\ell+\frac{3}{2})}.
\end{align*}
Since $P_n^{(\alpha,\beta)}(x;c) = (-1)^nP_n^{(\beta,\alpha)}(-x;c)$, the connection coefficients between $\{P_n^{(\frac{1}{2}, \beta)}(x;c)\}_{n=0}^\infty$ and $\{P_\ell^{(\frac{1}{2}, \beta)}(x)\}_{\ell=0}^\infty$ are similar.
\end{lemma}

If both source and target families satisfy $p_n(-x) = (-1)^np_n(x)$, then this symmetry imparts in the connection coefficients a chessboard pattern of zeros.

\begin{lemma}[Lewanowicz~\cite{Lewanowicz-49-137-93}]\label{lemma:associatedultraspherical}
The connection coefficients between the associated Jacobi polynomials $\{P_n^{(\alpha,\alpha)}(x;c)\}_{n=0}^\infty$ and the Jacobi polynomials $\{P_\ell^{(\alpha,\alpha)}(x)\}_{\ell=0}^\infty$ are given by:
\[
V_{\ell,n} = \left\{\begin{array}{ccc} u_n^{(\alpha, c)} t_{\ell,n}^{(\alpha, c)} h_{\ell,n}^{(\alpha, c)} v_\ell^{(\alpha)} & \for & \ell\le n,\quad \ell+n\hbox{ even},\\ 0 & \multicolumn{2}{c}{otherwise.}\end{array}\right.
\]
where:
\begin{align*}
u_n^{(\alpha, c)} & = \dfrac{\Gamma(n+c+\alpha+1)\Gamma(c+1)\Gamma(c+2\alpha+1)}{\Gamma(c+\alpha+1)\Gamma(n+c+1)\Gamma(c)},\\
t_{\ell,n}^{(\alpha, c)} & = \dfrac{\Gamma(\frac{1}{2}+\alpha)\Gamma(\frac{n-\ell+1}{2}-\alpha)\Gamma(\frac{n-\ell}{2}+c)}{\Gamma(\frac{n-\ell+1}{2}+c+\alpha)\Gamma(\frac{1}{2}-\alpha)\Gamma(\frac{n-\ell}{2}+1)},\\
h_{\ell,n}^{(\alpha, c)} & = \dfrac{\Gamma(\frac{n+\ell+2}{2})\Gamma(\frac{n+\ell+1}{2}+c+\alpha)}{\Gamma(\frac{n+\ell+3}{2}+\alpha)\Gamma(\frac{n+\ell+2}{2}+c+2\alpha)},\\
v_\ell^{(\alpha)} & = (\ell+\alpha+\tfrac{1}{2})\dfrac{\Gamma(\alpha+1)\Gamma(\ell+2\alpha+1)}{\Gamma(\ell+\alpha+1)\Gamma(2\alpha+1)}.
\end{align*}
\end{lemma}

\begin{lemma}[Askey and Wimp~\cite{Askey-Wimp-96-15-84}]\label{lemma:associatedhermite}
The connection coefficients between the associated Hermite polynomials $\{H_n(x;c)\}_{n=0}^\infty$ and the Hermite polynomials $\{H_\ell(x)\}_{\ell=0}^\infty$ are given by:
\[
V_{\ell,n} = \left\{\begin{array}{ccc} \dfrac{(-2)^{\frac{n-\ell}{2}}\Gamma(\frac{n-\ell}{2}+c)}{\Gamma(c)\Gamma(\frac{n-\ell}{2}+1)}\dfrac{\Gamma(\frac{n+\ell+2}{2})}{\Gamma(\ell+1)} & \for & \ell\le n,\quad \ell+n\hbox{ even},\\ 0 & \multicolumn{2}{c}{otherwise.}\end{array}\right.
\]
\end{lemma}

The dual purpose of this section is to suggest that other classes of fast transforms may be developed for these special cases. The lemmata above demonstrate that some special associated connection problems have a diagonally scaled Hadamard product structure between a Toeplitz and a Hankel matrix. This theoretically permits the fast factorization approach in~\cite{Townsend-Webb-Olver-87-1913-18}; however, in certain parameter r\'egimes, the two vectors defining the Toeplitz and Hankel parts grow and/or decay so rapidly that a numerical implementation would exhibit overflow and/or underflow, respectively. Explicit formul\ae~for the connection coefficients may also enable adaptations of the FMM, extending the approach in~\cite{Alpert-Rokhlin-12-158-91,Keiner-31-2151-09}, though similar problems of growth and decay are present in the analysis of the off-diagonal numerical rank of the subblocks.

\section{On the condition of the first associated Legendre--Legendre connection problem}\label{section:associatedLegendrecondition}

We have alluded to the fact that some associated connection problems may be ill-conditioned. It is of importance, then, to establish modest bounds on the condition number in at least one scenario.

By Lemma~\ref{lemma:associatedultraspherical}, we see that the first associated Legendre--Legendre connection coefficients are given by:
\[
V_{\ell,n} = \left\{\begin{array}{ccc} \dfrac{2(2\ell+1)}{(n-\ell+1)(n+\ell+2)} & \for & \ell\le n,\quad \ell+n\hbox{ even},\\ 0 & \multicolumn{2}{c}{otherwise.}\end{array}\right.
\]
The $2$-norm condition number, $\kappa_2(V)$, is equal to the ratio of the largest to the smallest singular values. It would have been too easy to estimate upper and lower bounds, respectively, for the largest and smallest singular values by the formul\ae~in~\cite{Qi-56-105-84,Johnson-112-1-89,Yi-Sheng-Dun-he-253-25-97,Johnson-Szulc-272-169-98}. However, due to the slow off-diagonal decay, strictly upper-triangular absolute row and column sums are unbounded as $n\to\infty$. In consequence, the best lower bounds for the smallest singular value are eventually $0$ for $n$ sufficiently large.

We turn to the theory of $M$-matrices~\cite{Willoughby-18-75-77,Plemmons-18-175-77,Johnson-47-195-82,Imam-55-93-83,Lewin-118-83-89,Johnson-Smith-435-953-11}. We wish to show that $V$ is an inverse $M$-matrix; that is, its spectrum is in the closed right-half complex plane and $V^{-1}$ has non-positive off-diagonal entries. Given that $V$ is a triangular matrix, it is clear that the spectra of $V$ and $V^{-1}$ are positive. Combining the non-positive off-diagonal property of $V^{-1}$ with triangular back substitution, it would follow from elementary row operations that:
\[
|(V^{-1})_{\ell,n}| \le \dfrac{V_{\ell,n}}{V_{\ell,\ell}V_{n,n}} = \left\{\begin{array}{ccc} \dfrac{(2\ell+2)(2n+2)}{2(2n+1)(n-\ell+1)(n+\ell+2)} & \for & \ell\le n,\quad \ell+n\hbox{ even},\\ 0 & \multicolumn{2}{c}{otherwise.}\end{array}\right.
\]
If this inequality were true, we would use H\"older's inequality~\cite{Golub-Van-Loan-13} for the condition number, $\kappa_2(V) \le \sqrt{\kappa_1(V)\kappa_\infty(V)}$, and estimate squared-logarithmic growth from each of $\kappa_1(V)$ and $\kappa_\infty(V)$, proving Theorem~\ref{theorem:associatedLegendrecondition}.

\begin{theorem}\label{theorem:associatedLegendrecondition}
For the associated Legendre--Legendre connection problem, $\kappa_2(V) = \OO(\log^2n)$.
\end{theorem}

\begin{proof}
Discretizing Eq.~\eqref{eq:ACOPEIG3b} results in the ``forced'' upper-triangular and banded generalized eigenvalue problem in Eq.~\eqref{eq:forcedGEP} where:
\[
A_{\ell,n} = \left\{\begin{array}{ccc} \dfrac{\ell(\ell+1)(\ell+2)}{2(2\ell+1)} & \for & \ell = n,\\ -\dfrac{(\ell+2)^2(\ell+3)}{2(2\ell+5)} & \for & \ell = n-2,\\ 0 & \multicolumn{2}{c}{otherwise,} \end{array}\right. \qquad
B_{\ell,n} = \left\{\begin{array}{ccc} \dfrac{\ell+2}{2(2\ell+1)} & \for & \ell = n,\\ -\dfrac{\ell+2}{2(2\ell+5)} & \for & \ell = n-2,\\ 0 & \multicolumn{2}{c}{otherwise,} \end{array}\right.
\]
and the diagonal matrices have entries $\Lambda_{n,n} = (n+1)(n+2)$ and $\Gamma_{n,n} = -(n+2)$.

It is important to also note that $A = B\Omega$, where $\Omega$ is a diagonal matrix with $\Omega_{n,n} = n(n+1)$. Multiplying Eq.~\eqref{eq:forcedGEP} by $B^{-1}$, we find:
\[
B^{-1}AV-V\Lambda = \Omega V - V\Lambda = B^{-1}\Gamma.
\]
Subsequent multiplication by $V^{-1}$ from the left and the right results in:
\[
\Lambda V^{-1} - V^{-1}\Omega = -V^{-1}B^{-1}\Gamma V^{-1} = (-V\Gamma^{-1}BV)^{-1}.
\]
Now, if $-V\Gamma^{-1}BV$ is an $M$-matrix, its inverse is non-negative. Given that the difference in diagonal scalings $\Lambda$ and $\Omega$ changes sign on the first super-diagonal, it would follow from the component-wise formula:
\[
(V^{-1})_{\ell,n} = \dfrac{[(-V\Gamma^{-1}BV)^{-1}]_{\ell,n}}{\Lambda_{\ell,\ell} - \Omega_{n,n}},\qquad \ell \ne n+1,
\]
that $V^{-1}$ is an $M$-matrix.

To begin, it is easy to show $-V\Gamma^{-1}BV$ has positive entries on the main diagonal. Next, Let $\ell \ge 0$ and $m \ge 1$. Then:
\begin{align*}
(-V\Gamma^{-1}BV)_{\ell,\ell+2m} &= \sum_{k=1}^m \left[\dfrac{2(2\ell+1)}{(2k+1)(2\ell+2k+2)} - \dfrac{2(2\ell+1)}{(2k-1)(2\ell+2k)}\right]\dfrac{1}{(2m-2k+1)(2\ell+2k+2m+2)}\\
&\qquad\qquad\qquad + \dfrac{2(2\ell+1)}{(2\ell+2)(2m+1)(2\ell+2m+2)}.
\end{align*}
This sum is non-positive if and only if:
\begin{align*}
&\sum_{k=1}^m \left[\dfrac{1}{(2k+1)(2\ell+2k+2)} - \dfrac{1}{(2k-1)(2\ell+2k)}\right]\dfrac{1}{(2m-2k+1)(2\ell+2k+2m+2)}\\
& \qquad\qquad\qquad + \dfrac{1}{(2\ell+2)(2m+1)(2\ell+2m+2)} \le 0.
\end{align*}
Or equivalently:
\[
\sum_{k=1}^m \left[\dfrac{\ell+1}{(2k-1)(k+\ell)} - \dfrac{\ell+1}{(2k+1)(k+\ell+1)}\right]\dfrac{(2m+1)(\ell+m+1)}{(2m-2k+1)(k+\ell+m+1)} \ge 1.
\]
Consider the infinite telescoping series:
\[
\sum_{k=1}^\infty \left[\dfrac{\ell+1}{(2k-1)(k+\ell)} - \dfrac{\ell+1}{(2k+1)(k+\ell+1)}\right] = 1.
\]
We use this clever form of unity to restate the inequality that we must prove as:
\begin{align*}
& \sum_{k=1}^m \left[\dfrac{\ell+1}{(2k-1)(k+\ell)} - \dfrac{\ell+1}{(2k+1)(k+\ell+1)}\right]\left[\dfrac{(2m+1)(\ell+m+1)}{(2m-2k+1)(k+\ell+m+1)}-1\right]\\
\ge & \sum_{k=m+1}^\infty\left[\dfrac{\ell+1}{(2k-1)(k+\ell)} - \dfrac{\ell+1}{(2k+1)(k+\ell+1)}\right].
\end{align*}
As the right-hand side is also telescoping, we must show:
\[
\sum_{k=1}^m \left[\dfrac{\ell+1}{(2k-1)(k+\ell)} - \dfrac{\ell+1}{(2k+1)(k+\ell+1)}\right]\left[\dfrac{(2m+1)(\ell+m+1)}{(2m-2k+1)(k+\ell+m+1)}-1\right] \ge \dfrac{\ell+1}{(2m+1)(\ell+m+1)}.
\]
Since:
\[
\dfrac{(2m+1)(\ell+m+1)}{(2m-2k+1)(k+\ell+m+1)}-1 = \dfrac{k(2k+2\ell+1)}{(2m-2k+1)(k+\ell+m+1)},
\]
we simplify (canceling $\ell+1$ from both sides):
\[
\sum_{k=1}^m \left[\dfrac{2k+2\ell+1}{(2k-1)(k+\ell)} - \dfrac{2k+2\ell+1}{(2k+1)(k+\ell+1)}\right]\dfrac{k}{(2m-2k+1)(k+\ell+m+1)} \ge \dfrac{1}{(2m+1)(\ell+m+1)}.
\]
Since:
\[
\dfrac{2k+2\ell+1}{k+\ell} > 2, \quad{\rm and}\quad \dfrac{2k+2\ell+1}{k+\ell+1} < 2,
\]
and:
\[
\dfrac{k}{2m-2k+1} \ge \dfrac{1}{2m-1} > \dfrac{1}{2m},
\]
the left-hand side is greater than or equal to:
\[
\sum_{k=1}^m \left(\dfrac{2}{2k-1} - \dfrac{2}{2k+1}\right)\dfrac{1}{2m(k+\ell+m+1)}.
\]
Finally, since:
\[
\dfrac{1}{k+\ell+m+1} \ge \dfrac{1}{\ell+2m+1},
\]
we again use the telescoping series to find that the left-hand side is greater than or equal to:
\[
\left(2-\dfrac{2}{2m+1}\right)\dfrac{1}{2m(\ell+2m+1)} = \dfrac{2}{(2m+1)(\ell+2m+1)}.
\]
The proof follows since:
\[
\dfrac{2\ell+2m+2}{\ell+2m+1} > 1.
\]
\end{proof}

Numerical evidence in Table~\ref{table:legendre_least_singular_value} suggests that the smallest singular value $\sigma_n(V)$ tends to a constant as $n\to\infty$, so that $\kappa_2(V) = \OO(\log n)$.

\begin{table}[htp]
\caption{The least singular value of the first associated Legendre -- Legendre connection coefficients.}
\begin{center}
\begin{tabular}{rcrc}
\hline
$n$ & $\sigma_n(V)$ & $n$ & $\sigma_n(V)$\\
\hline
$4$ & $0.9923428263553113$ & $128$ & $0.9923377094999823$\\
$8$ & $0.9923387019979827$ & $256$ & $0.9923377094922627$\\
$16$ & $0.9923377554186819$ & $512$ & $0.9923377094917538$\\
$32$ & $0.9923377118451174$ & $1024$ & $0.9923377094917182$\\
$64$ & $0.9923377096254247$ & $2048$ & $0.9923377094917155$\\
\hline
\end{tabular}
\end{center}
\label{table:legendre_least_singular_value}
\end{table}%

\section{The Hilbert transform}

Given a function $f\in L^2(D,\ud\mu)$, we consider its (weighted) Hilbert transform:
\[
{\cal H}_D\{f\}(x) = \frac{1}{\pi}\dashint_D \frac{f(t)}{t-x}\ud\mu(t),\quad{\rm for}\quad x\in D,
\]
where the dashed integral is interpreted as a Cauchy principal value. Applications of the Hilbert transform arise as a consequence of it being the solution operator to certain Riemann--Hilbert problems~\cite{Trogdon-Olver-16}. The obvious algorithm to compute the Hilbert transform is to use singular integral quadrature rule. These are generally useful for evaluation at a single point. But such schemes cannot rapidly evaluate the weighted Hilbert transform of a degree-$(n-1)$ polynomial at $n$ points in $\OO(n\log n)$ flops.

A common alternative strategy~\cite{King-1-09,King-2-09} to compute Hilbert transforms is to cleave the singularity as:
\[
{\cal H}_D\{f\}(x) = \frac{1}{\pi}\int_D \frac{f(t)-f(x)}{t-x}\ud\mu(t) + \frac{f(x)}{\pi}\dashint_D \frac{\ud\mu(t)}{t-x}.
\]
Thus, only the Hilbert transform of the measure requires principal value treatment.

From the cleaved representation, the weighted Hilbert transform of any finite orthogonal polynomial expansion:
\[
f(x) = \sum_{k=0}^{n-1} c_k p_k(x),
\]
is given by:
\begin{equation}\label{eq:HilbertTransformofpn}
{\cal H}_D\{f\}(x) = \frac{A_0\int_D\ud\mu(t)}{\pi}\sum_{k=0}^{n-2} c_{k+1}p_k(x;1) + \left(\sum_{k=0}^{n-1} c_kp_k(x)\right) \frac{1}{\pi}\dashint_D\frac{\ud\mu(t)}{t-x}
\end{equation}
Solving the connection problem between associated orthogonal polynomials $p_k(x;1)$ and the original polynomials $p_k(x)$, we can work with a common basis. With fast synthesis with $p_k(x)$, this process enables the rapid computation of the Hilbert transform on the same grid, provided we compute the Hilbert transform of the measure.

\begin{figure}[htbp]
\begin{center}
\begin{tabular}{cc}
\hspace*{-0.2cm}\includegraphics[width=0.53\textwidth]{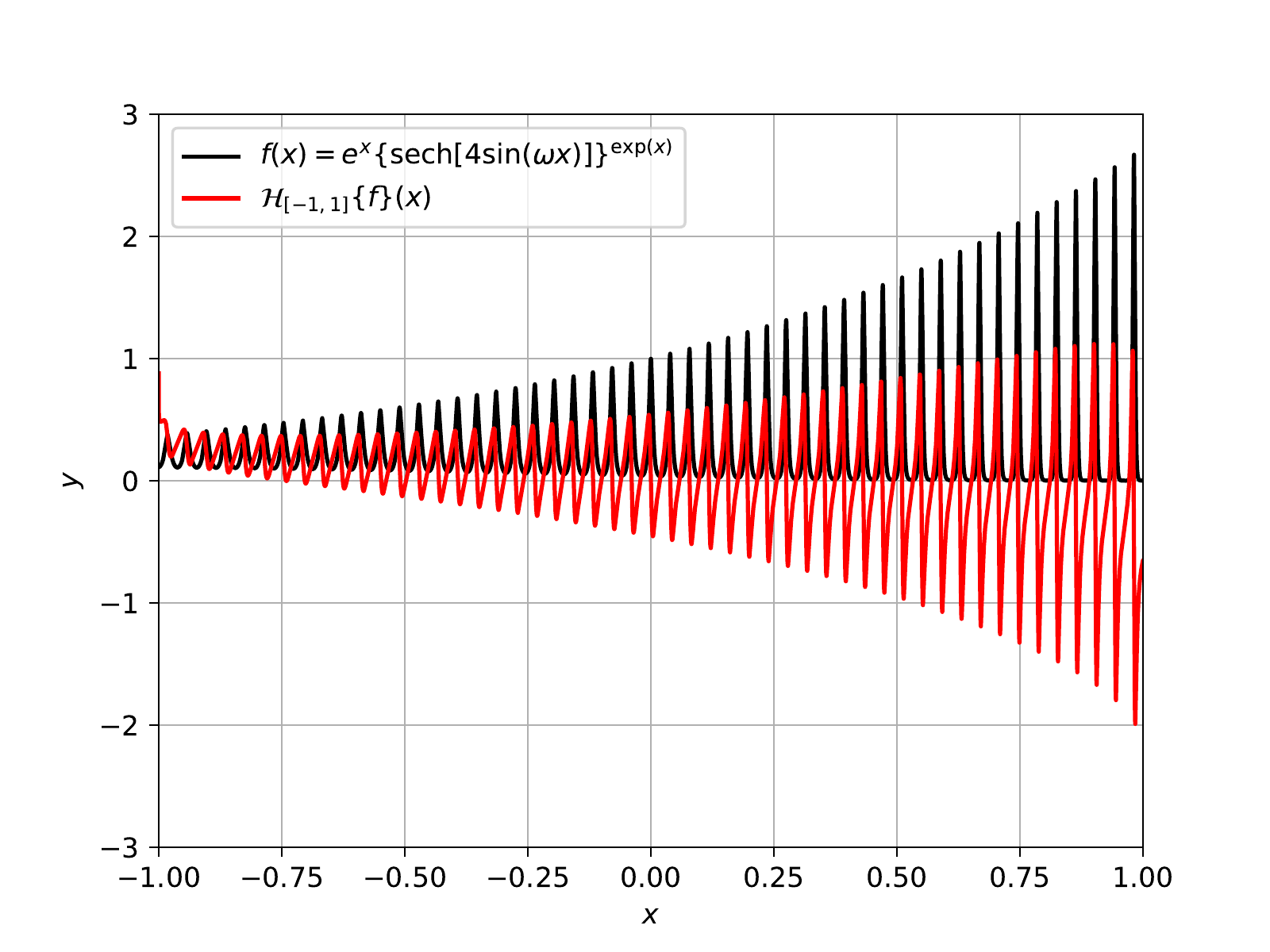}&
\hspace*{-0.65cm}\includegraphics[width=0.53\textwidth]{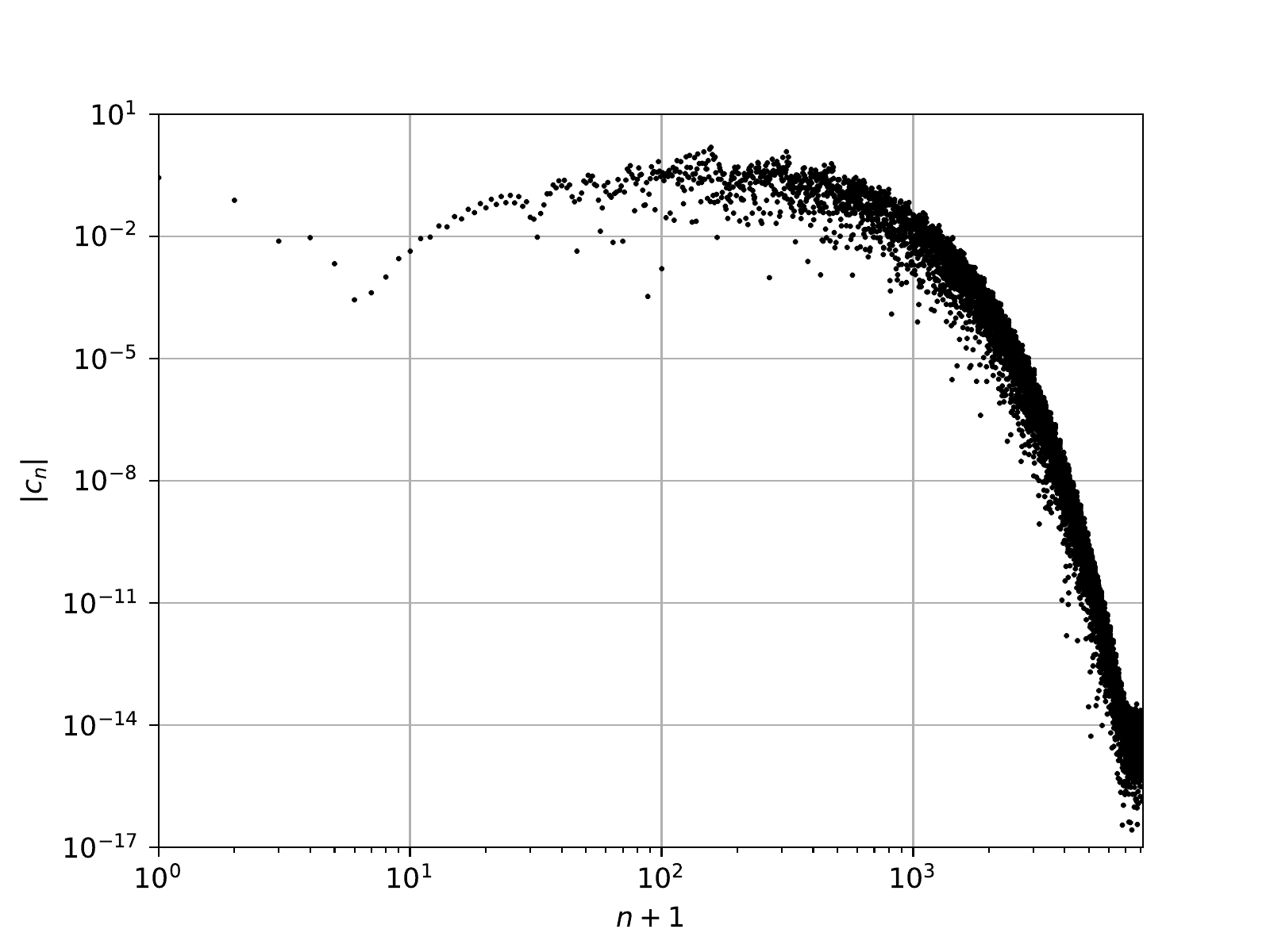}\\
\end{tabular}
\caption{Left: a function and its Hilbert transform, $\omega = 80$ and $n = 8192$. Right: the approximate Legendre coefficients of $f(x)$.}
\label{fig:hilbertplotandcoefficients}
\end{center}
\end{figure}

Figure~\ref{fig:hilbertplotandcoefficients} shows a particular function and its Hilbert transform on the unit interval with the uniform measure, $\ud\mu = \ud x$, by sampling $f$ on a Chebyshev grid, analyzing it in a Chebyshev series, converting the Chebyshev series to a Legendre series, and using Eq.~\eqref{eq:HilbertTransformofpn}.

There are other strategies for the unit interval. With a uniform measure, Olver~\cite{Olver-80-1745-11} uses the Joukowsky transform to map the unit interval to the unit circle in the complex plane, and identifies a set of special functions that incorporate the inherent discontinuity in the uniform measure mapped to the unit circle. With a non-negatively weighted measure, $\ud\mu = w(x)\ud x$, Hasegawa and Torii~\cite{Hasegawa-Torii-51-71-94} expand $f(x)$ in Chebyshev polynomials of the first kind:
\[
f(x) = \sum_{k=0}^{n-1}c_k^{\rm T} T_k(x),
\]
and, using the formula due to Elliott~\cite[Appendix 1]{Elliott-19-234-65}:
\[
\dfrac{T_{k+1}(x)-T_{k+1}(t)}{x-t} = U_k(t) + 2\sum_{j=1}^k U_{k-j}(t)T_j(x),
\]
find:
\[
{\cal H}_D\{f\}(x) = \frac{1}{\pi}\sum_{k=0}^{n-2}c_{k+1}^{\rm T} \int_{-1}^1 \left(U_k(t) + 2\sum_{j=1}^kU_{k-j}(t)T_j(x)\right) w(t)\ud t + \left(\sum_{k=0}^{n-1}c_k^{\rm T}T_k(x)\right) \frac{1}{\pi}\dashint_{-1}^1 \frac{w(t)\ud t}{t-x}.
\]
If:
\[
w_j = \int_{-1}^1U_j(t)w(t)\ud t,
\]
then by reversing the order of summation:
\[
{\cal H}_D\{f\}(x) = \frac{1}{\pi}\left(\sum_{k=0}^{n-2}c_{k+1}^{\rm T} w_k + 2\sum_{j=1}^{n-2}T_j(x)\sum_{k=j}^{n-2} c_{k+1}^{\rm T}w_{k-j}\right) + \left(\sum_{k=0}^{n-1}c_k^{\rm T}T_k(x)\right) \frac{1}{\pi}\dashint_{-1}^1 \frac{w(t)\ud t}{t-x}.
\]
The discrete convolutions can be cast as an upper-triangular Toeplitz matrix-vector product, which can be applied in $\OO(n\log n)$ operations via the fast Fourier transform.

\section{Conclusions}

We have developed fast approximate solutions to the associated classical -- classical orthogonal polynomial connection problem based on the differential Eqs.~\eqref{eq:ACOPEIG3a} and~\eqref{eq:ACOPEIG2a}. We have described when we anticipate these solutions to be successful and when the ill-conditioning of the problem warrants the development of alternative approaches. Promising alternatives require the fast approximate solution of structured nonnormal matrix equations, a challenging area of active research.

We have not fully explored the use of the differential Eqs.~\eqref{eq:ACOPEIG1},~\eqref{eq:ACOPEIG2},~\eqref{eq:ACOPEIG3}, and~\eqref{eq:ACOPEIG4} for fast transforms. Of these, Eqs.~\eqref{eq:ACOPEIG1} and~\eqref{eq:ACOPEIG2} are essentially the same and Eq.~\eqref{eq:ACOPEIG4} seems to be the least likely candidate for success: the polynomial variable coefficients of the two factors in Eq.~\eqref{eq:ACOPEIG4} have degrees at most $c+1$ and $2c$, respectively, the degree-expanding nature of which would create nontrivial lower bands in the discretizations. The lower bands can be dealt with by applying a sufficiently high-order differential operator to ensure each factor is degree-preserving. However, the bandwidths of these differential discretizations depend on the order $c$ of association, tying it to the complexity of any algorithm whose complexity depends on the bandwidth.

There are other non-classical connection problems that may be accelerated by identifying similar structural relationships. Semi-classical orthogonal polynomials are those polynomials orthogonal with respect to a weight function that satisfies a first-order linear homogeneous differential equation with a rational coefficient. It has been shown~\cite{Magnus-57-215-95} that the polynomials satisfy a second-order linear homogeneous differential equation with all three coefficients variable in $x$ and $n$. Rational measure modifications do not necessarily satisfy a differential equation (unless they are modifying classical measures) but the identities they satisfy~\cite{Uvarov-9-1253-69} share enough properties to enable a structured solution of the connection problem.

\bibliography{/Users/Mikael/Bibliography/Mik}

\appendix

\section{Jacobi Recurrence Relations}\label{Appendix:JacobiRecurrences}

It will be helpful to set $\gamma = \alpha+\beta+1$ to simplify the following formul\ae.

\begin{proposition}
\begin{equation}
\DD\begin{pmatrix} P_0^{(\alpha,\beta)} & P_1^{(\alpha,\beta)} & \cdots\end{pmatrix}
=
\begin{pmatrix} P_0^{(\alpha+1,\beta+1)} & P_1^{(\alpha+1,\beta+1)} & \cdots\end{pmatrix}D_{(\alpha,\beta)}^{(\alpha+1,\beta+1)}
\end{equation}
where:
%\begin{equation}
%D_{(\alpha,\beta)}^{(\alpha+1,\beta+1)} = \begin{pmatrix} 0 & \frac{\alpha+\beta+2}{2}\\ & & \frac{\alpha+\beta+3}{2}\\ & & & \frac{\alpha+\beta+4}{2}\\ & & & & \ddots
%\end{pmatrix}.
%\end{equation}
\begin{equation}
D_{(\alpha,\beta)}^{(\alpha+1,\beta+1)} = \frac{1}{2}\begin{pmatrix} 0 & \gamma+1\\ & & \gamma+2\\ & & & \gamma+3\\ & & & & \ddots
\end{pmatrix}.
\end{equation}
\end{proposition}

\begin{proposition}
\begin{equation}
x\begin{pmatrix} P_0^{(\alpha,\beta)} & P_1^{(\alpha,\beta)} & \cdots\end{pmatrix}
=
\begin{pmatrix} P_0^{(\alpha,\beta)} & P_1^{(\alpha,\beta)} & \cdots\end{pmatrix}M_{(\alpha,\beta)}
\end{equation}
where:
%\begin{equation}
%M_{(\alpha,\beta)} = \begin{pmatrix} \frac{\beta-\alpha}{\alpha+\beta+2} & \frac{2(\alpha+1)(\beta+1)}{(\alpha+\beta+2)(\alpha+\beta+3)} \\ \frac{2}{\alpha+\beta+2} & \frac{\beta^2-\alpha^2}{(\alpha+\beta+2)(\alpha+\beta+4)} & \frac{2(\alpha+2)(\beta+2)}{(\alpha+\beta+4)(\alpha+\beta+5)}\\ & \frac{4(\alpha+\beta+2)}{(\alpha+\beta+3)(\alpha+\beta+4)} & \frac{\beta^2-\alpha^2}{(\alpha+\beta+4)(\alpha+\beta+6)} & \frac{2(\alpha+3)(\beta+3)}{(\alpha+\beta+6)(\alpha+\beta+7)}\\ & & \frac{6(\alpha+\beta+3)}{(\alpha+\beta+5)(\alpha+\beta+6)} & \frac{\beta^2-\alpha^2}{(\alpha+\beta+6)(\alpha+\beta+8)} & \ddots\\
%& & & \ddots & \ddots
%\end{pmatrix}.
%\end{equation}
\begin{equation}
M_{(\alpha,\beta)} = \begin{pmatrix} \frac{\beta-\alpha}{\gamma+1} & \frac{2(\alpha+1)(\beta+1)}{(\gamma+1)(\gamma+2)} \\ \frac{2}{\gamma+1} & \frac{\beta^2-\alpha^2}{(\gamma+1)(\gamma+3)} & \frac{2(\alpha+2)(\beta+2)}{(\gamma+3)(\gamma+4)}\\ & \frac{4(\gamma+1)}{(\gamma+2)(\gamma+3)} & \frac{\beta^2-\alpha^2}{(\gamma+3)(\gamma+5)} & \frac{2(\alpha+3)(\beta+3)}{(\gamma+5)(\gamma+6)}\\ & & \frac{6(\gamma+2)}{(\gamma+4)(\gamma+5)} & \frac{\beta^2-\alpha^2}{(\gamma+5)(\gamma+7)} & \ddots\\
& & & \ddots & \ddots
\end{pmatrix}.
\end{equation}
\end{proposition}

\begin{proposition}
\begin{equation}
\begin{pmatrix} P_0^{(\alpha,\beta)} & P_1^{(\alpha,\beta)} & \cdots\end{pmatrix}
=
\begin{pmatrix} P_0^{(\alpha+1,\beta+1)} & P_1^{(\alpha+1,\beta+1)} & \cdots\end{pmatrix}R_{(\alpha,\beta)}^{(\alpha+1,\beta+1)}
\end{equation}
where:
%\begin{equation}
%R_{(\alpha,\beta)}^{(\alpha+1,\beta+1)} = \begin{pmatrix} 1 & \frac{(\alpha-\beta)(\alpha+\beta+2)}{(\alpha+\beta+2)(\alpha+\beta+4)} & -\frac{(\alpha+2)(\beta+2)}{(\alpha+\beta+4)(\alpha+\beta+5)}\\ & \frac{(\alpha+\beta+2)(\alpha+\beta+3)}{(\alpha+\beta+3)(\alpha+\beta+4)} & \frac{(\alpha-\beta)(\alpha+\beta+3)}{(\alpha+\beta+4)(\alpha+\beta+6)} & -\frac{(\alpha+3)(\beta+3)}{(\alpha+\beta+6)(\alpha+\beta+7)}\\ & & \frac{(\alpha+\beta+3)(\alpha+\beta+4)}{(\alpha+\beta+5)(\alpha+\beta+6)} & \frac{(\alpha-\beta)(\alpha+\beta+4)}{(\alpha+\beta+6)(\alpha+\beta+8)} & -\frac{(\alpha+4)(\beta+4)}{(\alpha+\beta+8)(\alpha+\beta+9)}\\ & & & \ddots & \ddots & \ddots\end{pmatrix}.
%\end{equation}
\begin{equation}
R_{(\alpha,\beta)}^{(\alpha+1,\beta+1)} = \begin{pmatrix} 1 & \frac{(\alpha-\beta)(\gamma+1)}{(\gamma+1)(\gamma+3)} & -\frac{(\alpha+2)(\beta+2)}{(\gamma+3)(\gamma+4)}\\ & \frac{(\gamma+1)(\gamma+2)}{(\gamma+2)(\gamma+3)} & \frac{(\alpha-\beta)(\gamma+2)}{(\gamma+3)(\gamma+5)} & -\frac{(\alpha+3)(\beta+3)}{(\gamma+5)(\gamma+6)}\\ & & \frac{(\gamma+2)(\gamma+3)}{(\gamma+4)(\gamma+5)} & \frac{(\alpha-\beta)(\gamma+3)}{(\gamma+5)(\gamma+7)} & -\frac{(\alpha+4)(\beta+4)}{(\gamma+7)(\gamma+8)}\\ & & & \ddots & \ddots & \ddots\end{pmatrix}.
\end{equation}
\end{proposition}

\begin{proposition}
\begin{equation}
(1-x^2)\begin{pmatrix} P_0^{(\alpha+1,\beta+1)} & P_1^{(\alpha+1,\beta+1)} & \cdots\end{pmatrix}
=
\begin{pmatrix} P_0^{(\alpha,\beta)} & P_1^{(\alpha,\beta)} & \cdots\end{pmatrix}L_{(\alpha+1,\beta+1)}^{(\alpha,\beta)}
\end{equation}
where:
%\begin{equation}
%L_{(\alpha+1,\beta+1)}^{(\alpha,\beta)} = \begin{pmatrix} \frac{4(\alpha+1)(\beta+1)}{(\alpha+\beta+2)(\alpha+\beta+3)}\\ \frac{4(\alpha-\beta)}{(\alpha+\beta+2)(\alpha+\beta+4)} & \frac{4(\alpha+2)(\beta+2)}{(\alpha+\beta+4)(\alpha+\beta+5)}\\ -\frac{4(0+1)(0+2)}{(\alpha+\beta+3)(\alpha+\beta+4)} & \frac{8(\alpha-\beta)}{(\alpha+\beta+4)(\alpha+\beta+6)} & \frac{4(\alpha+3)(\beta+3)}{(\alpha+\beta+6)(\alpha+\beta+7)}\\ & -\frac{4(1+1)(1+2)}{(\alpha+\beta+5)(\alpha+\beta+6)} & \frac{12(\alpha-\beta)}{(\alpha+\beta+6)(\alpha+\beta+8)} & \ddots\\ & & -\frac{4(2+1)(2+2)}{(\alpha+\beta+7)(\alpha+\beta+8)} & \ddots\\ & & & \ddots\end{pmatrix}.
%\end{equation}
\begin{equation}
L_{(\alpha+1,\beta+1)}^{(\alpha,\beta)} = \begin{pmatrix} \frac{4(\alpha+1)(\beta+1)}{(\gamma+1)(\gamma+2)}\\ \frac{4(\alpha-\beta)}{(\gamma+1)(\gamma+3)} & \frac{4(\alpha+2)(\beta+2)}{(\gamma+3)(\gamma+4)}\\ -\frac{4(0+1)(0+2)}{(\gamma+2)(\gamma+3)} & \frac{8(\alpha-\beta)}{(\gamma+3)(\gamma+5)} & \frac{4(\alpha+3)(\beta+3)}{(\gamma+5)(\gamma+6)}\\ & -\frac{4(1+1)(1+2)}{(\gamma+4)(\gamma+5)} & \frac{12(\alpha-\beta)}{(\gamma+5)(\gamma+7)} & \ddots\\ & & -\frac{4(2+1)(2+2)}{(\gamma+6)(\gamma+7)} & \ddots\\ & & & \ddots\end{pmatrix}.
\end{equation}
\end{proposition}

The four operators can be composed naturally. For example, the second derivative results in $D_{(\alpha,\beta)}^{(\alpha+2,\beta+2)} = D_{(\alpha+1,\beta+1)}^{(\alpha+2,\beta+2)}D_{(\alpha,\beta)}^{(\alpha+1,\beta+1)}$.

\section{Laguerre Recurrence Relations}\label{Appendix:LaguerreRecurrences}

\begin{proposition}
\begin{equation}
\DD\begin{pmatrix} L_0^{(\alpha)} & L_1^{(\alpha)} & \cdots\end{pmatrix}
=
\begin{pmatrix} L_0^{(\alpha+1)} & L_1^{(\alpha+1)} & \cdots\end{pmatrix}D_{(\alpha)}^{(\alpha+1)}
\end{equation}
where:
\begin{equation}
D_{(\alpha)}^{(\alpha+1)} = \begin{pmatrix} 0 & -1\\ & & -1\\ & & & \ddots
\end{pmatrix}.
\end{equation}
\end{proposition}

\begin{proposition}
\begin{equation}
x\begin{pmatrix} L_0^{(\alpha)} & L_1^{(\alpha)} & \cdots\end{pmatrix}
=
\begin{pmatrix} L_0^{(\alpha)} & L_1^{(\alpha)} & \cdots\end{pmatrix}M_{(\alpha)}
\end{equation}
where:
\begin{equation}
M_{(\alpha)} = \begin{pmatrix} \alpha+1 & -\alpha-1 \\ -1 & \alpha+3 & -\alpha-2\\ & -2 & \alpha+5 & -\alpha-3\\ & & -3 & \alpha+7 & \ddots\\
& & & \ddots & \ddots
\end{pmatrix}.
\end{equation}
\end{proposition}

\begin{proposition}
\begin{equation}
\begin{pmatrix} L_0^{(\alpha)} & L_1^{(\alpha)} & \cdots\end{pmatrix}
=
\begin{pmatrix} L_0^{(\alpha+1)} & L_1^{(\alpha+1)} & \cdots\end{pmatrix}R_{(\alpha)}^{(\alpha+1)}
\end{equation}
where:
\begin{equation}
R_{(\alpha)}^{(\alpha+1)} = \begin{pmatrix} 1 & -1\\ & 1 & -1\\ & & \ddots & \ddots\end{pmatrix}.
\end{equation}
\end{proposition}

\begin{proposition}
\begin{equation}
x\begin{pmatrix} L_0^{(\alpha+1)} & L_1^{(\alpha+1)} & \cdots\end{pmatrix}
=
\begin{pmatrix} L_0^{(\alpha)} & L_1^{(\alpha)} & \cdots\end{pmatrix}L_{(\alpha+1)}^{(\alpha)}
\end{equation}
where:
\begin{equation}
L_{(\alpha+1)}^{(\alpha)} = \begin{pmatrix} \alpha+1\\ -1 & \alpha+2\\ & -2 & \alpha+3\\ & & -3 & \ddots\\ & & & \ddots\end{pmatrix}.
\end{equation}
\end{proposition}

The four operators can be composed naturally. For example, the second derivative results in $D_{(\alpha)}^{(\alpha+2)} = D_{(\alpha+1)}^{(\alpha+2)}D_{(\alpha)}^{(\alpha+1)}$. It is also true that $M_{(\alpha)} = L_{(\alpha+1)}^{(\alpha)}R_{(\alpha)}^{(\alpha+1)}$.

\end{document}

%% file: header.tex
\usepackage[T1]{fontenc}
\usepackage{times}
\usepackage{epsfig,graphicx,tabularx,color}
\usepackage[cp850]{inputenc}
\usepackage[english]{babel}
\usepackage{cite}
\usepackage{amsmath,amssymb,amsfonts,amsthm,mathrsfs,comment,mathdots,multirow,tikz}
\usepackage{rotating, longtable, relsize}
\usepackage[f]{esvect}
\usepackage{times}
\usepackage{fancyhdr}
\usepackage{lipsum}

\usetikzlibrary{calc,3d}

\newtheorem{theorem}{Theorem}[section]

\newtheorem{lemma}[theorem]{Lemma}

\newtheorem{proposition}[theorem]{Proposition}

\theoremstyle{definition}
\newtheorem{remark}[theorem]{Remark}

\def\sphline{\noalign{\vskip3pt}\hline\noalign{\vskip3pt}}

\def\Xint#1{\mathchoice
   {\XXint\displaystyle\textstyle{#1}}%
   {\XXint\textstyle\scriptstyle{#1}}%
   {\XXint\scriptstyle\scriptscriptstyle{#1}}%
   {\XXint\scriptscriptstyle\scriptscriptstyle{#1}}%
   \!\int}
\def\XXint#1#2#3{{\setbox0=\hbox{$#1{#2#3}{\int}$}
     \vcenter{\hbox{$#2#3$}}\kern-.5\wd0}}

\def\dashint{\Xint-}

\def\ud{{\rm\,d}}

\def\C{\mathbb{C}}

\def\N{\mathbb{N}}

\def\R{\mathbb{R}}

\def\Z{\mathbb{Z}}

\def\AA{\mathcal{A}}
\def\BB{\mathcal{B}}

\def\DD{\mathcal{D}}

\def\II{\mathcal{I}}

\def\OO{\mathcal{O}}

\def\pr(#1){\left({#1}\right)}
\def\br[#1]{\left[{#1}\right]}

\def\abs#1{\left|{#1}\right|}
\def\norm#1{\left\|{#1}\right\|}

\def\for{\hbox{ for }}

\newcommand{\diag}{\operatorname{diag}}
\newcommand{\rank}{\operatorname{rank}}

\newcommand{\supp}{\operatorname{supp}}

\newcommand{\arginf}{\operatornamewithlimits{arg\,inf}}

\newcommand{\bs}[1]{\boldsymbol{#1}}